\documentclass[11pt,twoside]{amsart}
\usepackage{amsmath,amsfonts,amssymb,amscd,amsthm,amsbsy,epsf}
\usepackage{pdfsync}
\usepackage{setspace}
\usepackage{graphicx} 
\DeclareGraphicsExtensions{ .tiff, .jpg, .png}
\usepackage{epstopdf} 
\usepackage{fontenc}
\usepackage{esint}
\usepackage{bm}
\usepackage{color}
\usepackage{comment}
\usepackage{cite}
\usepackage{fp}
\newtheorem{thm}{Theorem}[section]
\newtheorem{lem}[thm]{Lemma}
\newtheorem{ex}[thm]{Example}
\newtheorem{cor}[thm]{Corollary}
\newtheorem{prop}[thm]{Proposition}
\newtheorem{remark}[thm]{Remark}
\theoremstyle{definition}
\newtheorem{defn}[thm]{Definition}
\numberwithin{equation}{section}

\newcommand{\ep}{\epsilon}
\newcommand{\p}{\partial}
\newcommand{\Om}{\Omega}
\newcommand{\om}{\omega}

\newcommand{\D}{ \mbox{$\Delta$} }
\newcommand{\Ga}{\Gamma}
\newcommand{\var}{\varphi}
\newcommand{\Btau}{\mathbf{\tau}}

\newcommand{\Be}{\mathbf{e}}
\def\sfint{\mathop{\int\mkern-16mu \raise.15ex\hbox{$\scriptstyle\diagup$}}\nolimits}
\def\fint{\mathop{\int\mkern-19mu {\diagup}}\nolimits}
\def\ds{\displaystyle}

\newcommand{\R}{\mathbb{R}}
\newcommand{\N}{\mathbb{N}}
\def\real{{\mathbb R}}

\newcommand{\ueind}[1]{ \mbox{$u^{\epsilon}_{#1}$}}

\providecommand{\hspace}[1]{  \mbox{$H^{#1}$}}

\newcommand{\dist}{ \mbox{$\mathrm{dist}$} }

\renewcommand{\div}{ \mbox{$\mathrm{div}$} }
\newcommand{\tr}{\mbox{$\mathrm{Tr}$}}
\newcommand{\ppuccip}{\mbox{$\mathcal{M}^{+}$}}
\newcommand{\ppuccin}{\mbox{$\mathcal{M}^{-}$}}

\newcommand{\puccip}{\mbox{$\mathcal{F}^{+}$}}
\newcommand{\puccin}{\mbox{$\mathcal{F}^{-}$}}
\newcommand{\pp}{\mbox{$\mathcal{M}^{+}$}}
\newcommand{\pn}{\mbox{$\mathcal{M}^{-}$}}
\newcommand{\Fp}{\mbox{$\mathcal{F}^{+}$}}
\newcommand{\Fn}{\mbox{$\mathcal{F}^{-}$}}

\def\C{{\mathcal C}}
\newcommand{\holder}{H\mbox{$\ddot{\mathrm{o}}$}lder  }
\providecommand{\diff}[1]{  \mbox{$\mathrm{d#1}$}  }

\providecommand{\abs}[1]{\left\vert#1\right\vert}
\providecommand{\norm}[2]{\left\Vert#1\right\Vert_{#2}}

\begin{document}
\title[Regularity of interfaces for a Pucci type  segregation problem]{Regularity of interfaces for a Pucci type  segregation problem}
\author{L.  Caffarelli} 
\address{The University of Texas at Austin\\
Department of Mathematics -- RLM 8.100\\
2515 Speedway -- Stop C1200\\
Austin, TX~78712-1202, US}
\email{caffarel@math.utexas.edu}
\author{S. Patrizi}
\address{The University of Texas at Austin\\
Department of Mathematics -- RLM 8.100\\
2515 Speedway -- Stop C1200\\
Austin, TX~78712-1202, US}
\email{spatrizi@math.utexas.edu}
\author{V. Quitalo}
\address{CMUC, Department of Mathematics, University of Coimbra, 3001-501 Coimbra, Portugal}
\email{vquitalo@math.utexas.edu}
\author{M. Torres}
\address{Purdue University\\
Department of Mathematics\\150 N. University Street\\
West Lafayette\\ IN 47907-2067, US}
\email{torresm@math.purdue.edu}
\thanks{This work was partially supported by the Centre for Mathematics of the University of Coimbra -- UID/MAT/00324/2013, funded by the Portuguese Government through FCT/MEC and co-funded by the European Regional Development Fund through the Partnership Agreement PT2020.
}

\begin{abstract}
We show the existence of a Lipschitz viscosity solution $u$ in $\Omega$ to a system of fully nonlinear equations involving Pucci-type operators. We study the regularity of the interface $\partial \{ u> 0 \}\cap\Om$ and we show that the viscosity inequalities of the system imply, in the weak sense, the free boundary condition $u^{+}_{\nu_{+}} = u^{-}_{\nu_{-}}$, and hence $u$ is a solution to a two-phase free boundary problem. We show that we can apply the classical method of sup-convolutions developed by the first author in 
\cite{caffarelli_harnack_1987,caffarelli_harnack_1989}, and generalized by Wang \cite{wang_regularity_2000,wang_regularity_2002} and Feldman  \cite{Fel} to fully nonlinear operators, to conclude that the regular points in $\partial \{ u> 0 \}\cap\Om$  form an open set of class $C^{1,\alpha}$. A novelty in our problem is that we have different operators, $\puccip$ and $\puccin$,  on each side of the free boundary. In the particular case when these operators are the Pucci's extremal operators $\ppuccip$ and $\ppuccin$, our results provide an alternative approach to obtain the stationary limit
 of a segregation model of populations with nonlinear diffusion in \cite{quitalo_free_2013}. 
\end{abstract}

\maketitle
\pagestyle{plain}
\baselineskip=24pt

\singlespace

\section{Introduction}
The work in the present paper is motivated by the study of the regularity of the free boundary for a limit problem obtained  from a segregation model with nonlinear diffusion studied 
by the third author in  \cite{quitalo_free_2013}. In  the case of two populations, the model takes the form
\begin{equation}
\label{pucciproblem}
\left\{
\begin{split}
  \ppuccin  (u^\ep_1) = \frac{1}{\epsilon}  u_1^\ep u^\ep_2& \quad  \mbox{in} \: \Omega   \\
      \ppuccin  (u^\ep_2) = \frac{1}{\epsilon}  u^\ep_1 u^\ep_2 & \quad  \mbox{in} \: \Omega \\
   u^\ep_i  = f_{i}   \quad i=1,2, &\quad  \mbox{on} \:  \partial \Omega,
  \end{split}
  \right.
 \end{equation}
 where  $\Om$ is a bounded Lipschitz domain of $\R^n$,  $f_{1}$ and $f_2$ are  non-negative, non-zero, \holder  continuous function defined on $ \partial \Omega$, 
  with disjoint supports, 
$\ppuccin $   denotes the  negative Pucci's  extremal operator
that  will be described later.   The non-negative solution $\ueind{i}$, $i=1,2$ of \eqref{pucciproblem}  can be seen as  a density of the population $i,$ and  the parameter $ \frac{1}{\epsilon}>0$ characterizes the level of competition between species. 
In \cite{quitalo_free_2013} it is proven that along a subsequence, $\ueind{1}$  and $\ueind{2}$  converge uniformly in $\Om$, as $\ep\to0^+$, respectively to $u_1$ and $u_2$, 
non-negative locally Lipschitz functions, solutions of the following free boundary problem, for $i,j=1,2$,  
\begin{equation}\label{limitproblempucci}
\begin{cases}
\ppuccin (u_i)= 0 &\text{in }\{u_i>0\}\\
\ppuccin (u_i-u_j)\leq 0&\text{in }\Om\\
u_1u_2=0&\text{in }\Om\\
u_i=f_i&\text{on }\partial\Om.\\
\end{cases}
\end{equation}
Let $u:=u_1-u_2$, then  $u_1=u^+$,  $u_2=u^-$, where $u^+,\,u^-$  are respectively the positive and negative part of $u$, and system \eqref{limitproblempucci} can be rewritten in terms of $u$ as follows
\begin{equation}\label{limitproblempuccibis}
\begin{cases}
\ppuccin (u)= 0 &\text{in }\{u>0\}\\
\ppuccip (u)= 0 &\text{in }\{u<0\}\\
\ppuccin (u)\leq 0&\text{in }\Om\\
\ppuccip (u)\geq 0&\text{in }\Om\\
u=f&\text{on }\partial\Om,\\
\end{cases}
\end{equation}
where $f=f_1-f_2$ and $\ppuccip(u)=-\ppuccin(-u)$ is the positive Pucci's operator. 

In the present paper  we study  problems likewise \eqref{limitproblempuccibis} in a more general setting. Precisely, we consider the  following free boundary problem,
\begin{equation}\label{limitproblemk=2}
\begin{cases}
\puccin (u)= 0 &\text{in }\{u>0\}\\
\puccip (u)= 0 &\text{in }\{u<0\} \\
\puccin (u)\leq 0&\text{in }\Om\\
\puccip(u) \geq 0&\text{in }\Om\\
u=f&\text{on }\partial\Om\\
\end{cases}
\end{equation}
in a bounded smooth domain $\Om\subset\real^n$, where $f$ is a Lipschitz function defined on $\partial\Om$ and  $\puccin $  and $\puccip$ 
are  uniformly elliptic operators  belonging to a class of extremal operators that includes the Pucci's operators $\ppuccin $  and $\ppuccip$. 
Therefore the limit problem \eqref{limitproblempuccibis} can be seen as a particular case of \eqref{limitproblemk=2}.

We first prove the existence of a  Lipschitz solution $u$ of \eqref{limitproblemk=2}. Then, we study the regularity of the free boundary set
$$\Gamma:=\partial\{u>0\}\cap \Om.$$
Denote $u_1=u^+$ and $u_2=u^-$, and  let $\nu_i$ be the interior unit normal vector to $\{u_i>0\}$. 
At this stage we have no information about the regularity of the free boundary  $\Gamma$ and the vectors $\nu_i$ may not be defined  at every point of $\Gamma$.
However, we can prove 
 that any Lipschitz solution of \eqref{limitproblemk=2} satisfies in a weak sense (viscosity sense) the following free boundary condition
\begin{equation*}
\frac{\partial u_1}{\partial \nu_1}=\frac{\partial u_2}{\partial \nu_2}\quad \text{on }\Gamma,\end{equation*}
that is, the normal derivative of $u$ is continuous across the free boundary. 

This will allow us to  apply the regularity  theory developed by  Caffarelli  in the  papers 
\cite{caffarelli_harnack_1986, caffarelli_harnack_1987, caffarelli_harnack_1989} for free boundary problems associated to linear operators and then extended  by  Wang \cite{wang_regularity_2000,wang_regularity_2002} to the case of fully nonlinear uniformly elliptic concave operators, to show that  the subset of regular points of the free boundary
is relatively  open in  $\Gamma$ and locally of class 
$C^{1,\alpha}$, $0<\alpha\le1$. 

Let us describe more in details the results of the present paper and the strategies followed.
Let $x_0\in\Gamma$ and assume that $\Gamma$ is smooth around $x_0$, 
then since $u$ is a  viscosity solution of  the first and second equation in \eqref{limitproblemk=2}, by the Hopf Lemma we  have
$$0<\frac{\partial u_1}{\partial \nu_1}(x_0),\,\frac{\partial u_2}{\partial \nu_2}(x_0)<+\infty,$$
that is $u$ has linear growth away from the free boundary around $x_0$. Thus, we expect that at points  where the solution $u$ "behaves well", in fact  both $u_1$ and $u_2$ have locally linear growth away from the free boundary.
The linear behavior of $u_i$ at a point $x_0$ of the free boundary without regularity assumptions on $\Gamma$ can be defined  as follows: there exists $\tilde{r}=\tilde{r}(x_0)>0$ 
and $M=M(x_0)>0$ 
such that for any $0<r<\tilde r$, 
\begin{equation}\label{introlinearbeha}\sup_{B_r(x_0)} u_i\ge M r.\end{equation}
A barrier argument shows that  the function $u_i$ satisfies \eqref{introlinearbeha}  at  points of $\Gamma$ where there is a tangent ball to $\Gamma$ contained in its support, as we will see. 
Points with this property are dense in $\Gamma$. 
Thus, we define   $x_0\in\Gamma$  to be regular if  
\eqref{introlinearbeha} holds true for  at least one among  $u_1$ and 
$u_2$, see Definition \ref{firstdefinition}.
Then by using that $u$ satisfies in the viscosity sense 
\begin{equation}\label{introlinearbehapucci}\puccin (u)\leq 0\leq \puccip (u)\quad\text{in }\Om\end{equation}
we can actually prove that both $u_1$ and  $ u_2$ have  linear behavior at any   regular point, as expected. The viscosity inequalities
\eqref{introlinearbehapucci} have to be understood as a sort of free boundary conditions since they are satisfied in the whole $\Om$ and thus across the free boundary too.

Now, solutions of \eqref{limitproblemk=2} have the properties that the positive and negative parts are subharmonic in $\Om$. Therefore, we can perform a blow up analysis 
by using   the monotonicity formula.  In particular, we can show that if $u$ is a Lipschitz solution of  \eqref{limitproblemk=2}, then 
 around any regular point  the free boundary is flat, meaning that it can be trapped in a narrow neighborhood in between two Lipschitz graphs. If  in addition there is a tangent ball from one side at $x_0\in\Gamma$, meaning that the ball is contained either in the positivity set of $u$ or in its negativity set, then we prove that $u$ has the asymptotic behavior
 \begin{equation}\label{asymptbehavintro}u(x)=\alpha<x-x_0,\nu>^+-\beta<x-x_0,\nu>^-+o(|x-x_0|),\end{equation}
 where $\alpha,\,\beta>0$ and $\nu$ is the normal vector to the tangent ball at $x_0$ pointing inward $\{u>0\}$.
 The viscosity inequalities \eqref{introlinearbehapucci} then imply $\alpha=\beta$, that is $u$ is  asymptotically a plane at $x_0$.
 This shows  that any Lipschitz viscosity solution of \eqref{limitproblemk=2} is also a viscosity solution to the following two phase free boundary problem
 \begin{equation}
\label{freebddproblemintro}
\begin{cases}
\puccin (u) =0\quad\text{in }\{u>0\}\\
\puccip (u) = 0 \quad\text{in } \{u<0\}\\
\frac{\partial u_1}{\partial \nu_1}=\frac{\partial u_2}{\partial \nu_2} \quad\text{on } \partial \{u>0\}\cap\Om.
 \end{cases}
 \end{equation}
We refer to \cite{caffarelli_geometric_2005} for the theory of viscosity solutions to free boundary problems. 
The regularity of the free boundary for problems of type  \eqref{freebddproblemintro} with same concave fully nonlinear operator  in both the positivity and the negativity set  of $u$
and  with more general free boundary conditions, has been investigated, as already mentioned,  in \cite{wang_regularity_2000,wang_regularity_2002}. More general  operators have been considered in \cite{MR2511639, MR3310271, Fel}.

Even though in  \eqref{freebddproblemintro}  there are  different operators on each side of the free boundary, we can still apply the results 
of \cite{wang_regularity_2000,wang_regularity_2002} and 
  prove that for any  solution $u$ of \eqref{freebddproblemintro} the following holds:
  if the free boundary is flat around a point $x_0\in\Gamma$, then in a neighborhood of $x_0$ it is  a $C^{1,\alpha}$ surface. Going back to the  original free boundary problem \eqref{limitproblemk=2}, this result implies that the set of regular points is an open subset of $\Gamma$ locally of class $C^{1,\alpha}$. In particular, $u$ has 
  the asymptotic behavior \eqref{asymptbehavintro} with $\alpha=\beta$ at any regular point.
  
To conclude, let us mention that we provide a simpler proof than in \cite{quitalo_free_2013} of the existence of a Lipschitz solution of \eqref{limitproblemk=2} that does not involve a segregation problem. Moreover as a byproduct of our results, we prove existence of a Lipschitz solution of \eqref{freebddproblemintro}. Existence of solutions to free boundary problems is in general a main issue. 
  For \eqref{freebddproblemintro}, with $\Fp$ replaced by $\Fn$ it has been proven in \cite{MR2005161}.  We believe that our existence proofs could be generalized  to a larger class of fully nonlinear operators. 
  
\subsection{Organization of the paper. }  
The operators $\Fn$ and $\Fp$  are defined and their properties described in Section \ref{Fnsec}. Some examples are provided too.
Our main results, Theorems \ref{existencethm}, \ref{limitpbthm} and \ref{C1alphaGammathm},  are contained in Section  \ref{mainresultssec}. In Section \ref{monotonicity} we recall the monotonicity formula and some related results.
Existence of a Lipschitz solution of the free boundary problem \eqref{freebddproblemintro}, i.e. Theorem \ref{existencethm}, is proven in Section \ref{Existencesec}. 
In Section \ref{regulapointssec} we introduce  the notion of regular points and we prove the non degeneracy of both $u_1$ and $u_2$ at regular points. 
Section \ref{Freeboundarycondisec}  is devoted to the proof of Theorem \ref{limitpbthm}. In Section \ref{C1alphasec} we prove that for the solution of \eqref{freebddproblemintro} flat free boundaries are Lipschitz and, as a corollary, Theorem \ref{C1alphaGammathm}.  Finally, some properties of the fundamental solution for the operator $\Fn$ are proven in the Appendix.

 \section{The operators $\Fn$ and $\Fp$. Notation}\label{Fnsec} 
 We will start by defining the two general fully nonlinear uniformly elliptic operators
 $\Fn$ and $\Fp$.  Let $\mathcal{S}_n$ be  the  set of symmetric $n \times n$ real matrices.
 Given $0<\lambda\leq 1<\Lambda$, let us denote by $\mathcal{A}_{\lambda, \Lambda}$ the set of matrices   of $\mathcal{S}_n$ with eigenvalues   in $[\lambda, \Lambda]$; i.e,
\begin{equation*}
\mathcal{A}_{\lambda, \Lambda}:=\{A\in  \mathcal{S}_n\,|\,\lambda I_n\leq A\leq\Lambda I_n\},
\end{equation*}
where $I_n$ is the identity matrix.
 Let $\mathcal{A}_{\lambda, \Lambda}^1$ and $\mathcal{A}_{\lambda, \Lambda}^2$ be two not empty subsets of $\mathcal{A}_{\lambda, \Lambda}$ with the property that 
 \begin{equation}\label{Fptope0}\text{ if }A \in \mathcal{A}_{\lambda, \Lambda}^{i},\, i=1,2,\text{ and }O\in\mathcal{O}(n),\,\text{ then } OA O^t\in \mathcal{A}_{\lambda, \Lambda}^i,
 \end{equation}
 where we denote by $\mathcal{O}(n)$ the set of $n\times n$ orthogonal matrices.
 Moreover, we assume that the identity matrix belongs to both sets,  
 \begin{equation}\label{identityinters}I_n \in \mathcal{A}_{\lambda, \Lambda}^1\cap \mathcal{A}_{\lambda, \Lambda}^2.\end{equation}
 Let $\puccip$ and $\puccin$  be the following operators defined over matrices $M$ in $\mathcal{S}_n$,
\begin{equation}
\label{op n}
\puccin(M):=\inf_{A \in \mathcal{A}_{\lambda, \Lambda}^1} \tr(AM )
\end{equation}
and
\begin{equation}
\label{op p}
\puccip(M):=\sup_{A \in \mathcal{A}_{\lambda, \Lambda}^2} \tr(AM).
\end{equation}
We remark that when $\mathcal{A}_{\lambda, \Lambda}^1=\mathcal{A}_{\lambda, \Lambda}^2=\mathcal{A}_{\lambda, \Lambda}$, then $\Fn=\pn$ and $\Fp=\pp$,  where $\pn$ and $\pp$ are the Pucci's extremal operators defined, for  $M\in\mathcal{S}_n$, as follows
$$\pn(M)=\inf_{A \in \mathcal{A}_{\lambda, \Lambda}} \tr(AM )=\lambda\sum_{e_i>0}e_i+\Lambda\sum_{e_i<0}e_i$$
and 
$$\pp(M)=\sup_{A \in \mathcal{A}_{\lambda, \Lambda} }\tr(AM )=\Lambda\sum_{e_i>0}e_i+\lambda\sum_{e_i<0}e_i,$$
where $e_i$, $i=1,\ldots,n$ are the eigenvalues of the matrix $M$.

\begin{prop}\label{Fnproperties}
 $\puccin$ and $\puccip$ satisfy, for $M,N\in \mathcal{S}_n$
 \begin{itemize}
\item [(a)] $\mathcal{F}^\pm(tM)=t\mathcal{F}^\pm(M)$ for any $t\geq 0$;
\item  [(b)]$\Fp(M+N)\leq \Fp(M)+\Fp(N)$ and hence $\Fp$ is convex;
\item  [(c)]$\Fn(M+N)\geq \Fn(M)+\Fn(N)$ and hence $\Fn$ is concave;
\item  [(d)] For any $M\in \mathcal{S}_n$,
$$\pn(M)\le\Fn(M)\leq \text{tr}(M)\leq \Fp(M)\leq \pp(M);$$
\item  [(e)] (Uniformly Ellipticity) $\pn(N)\leq\mathcal{F}^\pm(M+N)-\mathcal{F}^\pm(M) \leq  \ppuccip (N)$;
\end{itemize}
\end{prop}
\begin{proof}
Properties (a)-(c) are clear from the definitions \eqref{op n} and \eqref{op p} and the properties of the sup and inf functions. 

Since $ \mathcal{A}_{\lambda, \Lambda}^1,\,\mathcal{A}_{\lambda, \Lambda}^2\subset  \mathcal{A}_{\lambda, \Lambda}$, we have that  
$\pn(M)\le\Fn(M)$ and $\Fp(M)\leq \pp(M)$ for any $M\in\mathcal{S}_n$. Moreover, \eqref{identityinters} implies that 
$$\Fn(M)\leq  \text{tr}(M)\leq \Fp(M).$$ This proves (d).

By (b) and the last inequality in (d), we have that
$$ \Fp(M+N)-\Fp(M)\le \Fp(N)\leq  \pp(N).$$
On the other hand, by the properties of the sup function,
$$ \Fp(M+N)\ge \Fp(M)+\inf_{A \in \mathcal{A}^2_{\lambda, \Lambda} }\tr(AN )\ge \Fp(M)+\pn(N).$$
This concludes the proof of (e) for $\Fp$. Similarly, one can prove (e) for $\Fn$.
\end{proof}

Let $D$ be a  domain of $\R^n$. With a slight abuse of notation, we define the differential operators, for $u\in C^2(D)$ and $x\in D$,
$$\Fn(u)(x):=\Fn(D^2u(x)) $$ and 
$$\Fp(u)(x):=\Fp(D^2u(x)),$$ 
where $D^2u$ is the Hessian matrix of $u$.
By Proposition \ref{Fnproperties}, the differential operators $\Fn$ and $\Fp$ are 1-homogeneous, uniformly elliptic, $\Fn$ is concave and $\Fp$ is convex.
Moreover, by (d), for any $u\in C^2(D)$,
\begin{equation}\label{puccifoprel}\pn(u)\leq\Fn(u)\leq \Delta u\leq \Fp(u)\leq \pp(u),\end{equation} where again here we denote  $\pn(u)(x):=\pn(D^2 u(x))$,  $\pp(u)(x):=\pp(D^2 u(x))$ and by 
$\Delta u$ the Laplacian of $u$. Furthermore, the operators $\Fn$ and $\Fp$  are invariant under rotations, as stated in the following proposition.
\begin{prop}\label{orthogonalF}
Let $O$ be an orthogonal matrix. Let $u$ be a $C^2$-function and let $v(x)=u(Ox)$. Then,
$$\mathcal{F}^\pm(v)(x)=\mathcal{F}^\pm(u)(Ox).$$
\end{prop}
\begin{proof}
Since $D^2v(x)=O^tD^2u(Ox)O$, we have that 
\begin{equation*}\begin{split}\Fn(v)(x)&=\inf_{A\in\mathcal{A}_{\lambda, \Lambda}^1}\text{tr}\left(AO^tD^2u(Ox)O\right)=\inf_{A\in\mathcal{A}_{\lambda, \Lambda}^1}\text{tr}\left(OAO^tD^2u(Ox)\right)\\&
=\inf_{A\in\mathcal{A}_{\lambda, \Lambda}^1}\text{tr}\left(AD^2u(Ox)\right)=\Fn(u)(Ox),\end{split}\end{equation*}
where we have used that by \eqref{Fptope0},
$$\mathcal{A}_{\lambda, \Lambda}^1=\{OAO^t\,|\,A\in \mathcal{A}_{\lambda, \Lambda}^1,\, O\in\mathcal{O}(n)\}.$$

Similarly, $\Fp(v)(x)=\Fp(u)(Ox)$.
\end{proof}

\begin{remark}\label{FpFnremark}
By Proposition \ref{Fnproperties}, Harnack inequality holds true for nonnegative  viscosity solutions of $\puccin(u)\leq 0\leq \puccip (u),$ see \cite[Theorem 4.3]{ caffarelli_cabre_1995}. Observe also that $\Fn$ and $\Fp$ satisfy the  comparison principle:  if $D$ is a bounded domain and  $u$ is a viscosity subsolution for $\puccip$ in $D$, meaning $\puccip (u)\geq 0$ in the viscosity sense in $D$, $v$ is a viscosity supersolution for $\puccip$ in $D$, meaning $\puccip (v)\leq 0$ in the viscosity sense $D$, and $u\leq v$ on $\p D$ then $u\leq v$ in $\bar{D}$; the same result holds for $\Fn,$ 
see \cite{ caffarelli_cabre_1995, crandall_user_1992} for more details. In addition, since $\Fn$ and $\Fp$ are respectively concave and convex, interior $C^{2,\alpha}$-estimates for solutions of $\mathcal{F}^\pm(u)=0$ hold true, see  \cite{ caffarelli_cabre_1995}.
\end{remark}

\begin{remark}
\label{util} If $u$ is solution to \eqref{limitproblemk=2}, then 
\begin{equation*}\Gamma:=\partial\{u_1>0\}\cap\Om=\partial\{u_2>0\}\cap\Om.\end{equation*}
Indeed, if there was $x_0\in  \left( \partial\{u_1>0\}\cap\Om\right)\setminus\partial\{u_2>0\}$, then in a ball of radius $r$ around $x_0$ we would have 
$$ \puccin (u_1)=\puccin (u)\leq 0,\:u_1\geq 0,\: u_1\not\equiv 0,\: u_1(x_0)=0.$$
This contradicts the strong maximum principle.
 \end{remark}

\subsection{Some examples}
\begin{ex}
\label{ex: pucci}
 As discussed in the Introduction, the free boundary problem \eqref{limitproblempuccibis}, which is the  limit problem of a population model studied in 
 \cite{quitalo_free_2013}   that takes into account diffusion with preferential directions,  is a particular case of problem \eqref{limitproblemk=2}. Indeed by choosing $\mathcal{A}_{\lambda, \Lambda}^1=\mathcal{A}_{\lambda, \Lambda}^2=\mathcal{A}_{\lambda, \Lambda}$, we have that $\Fn=\pn$ and $\Fp=\pp$.
\end{ex}

\begin{ex}
By choosing $\mathcal{A}_{\lambda, \Lambda}^1=\{I_n\}$ and $\mathcal{A}_{\lambda, \Lambda}^2=\mathcal{A}_{\lambda, \Lambda}$, problem \eqref{limitproblemk=2} becomes
\begin{equation}\label{laplacianpuccipb}
\begin{cases}
\Delta u = 0 &\text{in }\{u>0\}\\
\ppuccip (u)= 0 &\text{in }\{u<0\} \\
\Delta u \leq 0&\text{in }\Om\\
\ppuccip(u) \geq 0&\text{in }\Om\\
u=f&\text{on }\partial\Om.\\
\end{cases}
\end{equation}
%
%
Since the Bellman equations are very helpful to solve optimal stopping strategies see \cite{krylov_controlled_2008}, this type of models can eventually be used to describe situations with multiple strategies.
\end{ex}

\begin{ex}
By the uniformly ellipticity, (e) in Proposition \ref{Fnproperties}, the operators $\Fp$ and $\Fn$  are Lipschitz continuous as functions in the space  $\mathcal{S}(n)$.
This regularity is optimal for the Pucci's operators $\ppuccin$ and $\ppuccip$. Indeed, 
 consider for example a family of matrices $\{M_t\,|\,t\in\R\}$ with eigenvalues $e_{1,t}=t$ and $e_{2,t}=e_{3,t}=\ldots=e_{n,t}=0$, then 
\begin{equation*}\ppuccin(M_t)=\begin{cases}\lambda t&\text{if }t\ge 0\\
\Lambda t&\text{if }t< 0,
\end{cases}
\end{equation*}
which is a no more than Lipschitz function for $\lambda<\Lambda$.
However there are operators in the class of extremal ones  that we consider here which are more regular. 
Consider for example,
$$\Fn(M)=\inf\{\text{tr}(AM)\,:\,A\in \mathcal{S}_p\}$$
$$\Fp(M)=\sup\{\text{tr}(AM)\,:\,A\in \mathcal{S}_p\}$$ where 
for $p>0$,
$$\mathcal{A}_{\lambda, \Lambda}^1=\mathcal{A}_{\lambda, \Lambda}^2=\mathcal{S}_p:=\{A=(a_{ij})\in \mathcal{S}_n \,:\,\|a_{ij}-\delta_{ij}\|_{l^p}\leq r_0<1\}$$
for some $\lambda=\lambda(r_0)<1\leq \Lambda$.  Since, for example, for $p= 2$ the balls in the $l^2$ norm are smooth, one can get a higher than Lipschitz regularity for $\Fp$ and $\Fn$ and thus, better than $C^{2,\alpha}$ estimates for the solutions $u$  of  $\mathcal{F}^\pm(u)=0$.
\end{ex}

\subsection{Notation} 
For a function $ u$,  $\nabla u$ and $D^2 u$ denote  respectively the gradient of $u$ and the Hessian matrix of  $u$. The standard Euclidean inner produt is denoted by $<\cdot, \cdot>$. We define $u^+:=\max(u,0)$ and $u^-:=\max(-u,0)$ which are the positive and negative part of $u$.
  In the rest of the paper, for the solution $u$ of  \eqref{limitproblemk=2}, we will use the notation 
  \begin{equation}\label{pos neg}
 u_1:=u^+ \quad \mbox{and} \quad u_2:=u^-,
 \end{equation} 
 at our convenience.
Notice that  
 \begin{equation*}
 |u|(x)=\max (u_1(x), u_2(x))=u_1(x)+u_2(x).
 \end{equation*}
Furthermore, we denote by 
$$\Om(u_i):=\{u_i>0\}$$
$i=1,2$,  the positivity set of $u_i$ and by 
 $$\Ga:=\partial\{u>0\}\cap\Om,$$
the free boundary set. 
If $u$ has an asymptotic development around $x_0$  along the direction $\nu_1$
given by,
$$
u(x)= \alpha <x-x_0, \nu_1>^+-\beta <x-x_0, \nu_1>^- + o(|x-x_0|)
$$ we write that 
$$
\frac{\p u_1}{\p \nu_1}=\alpha \quad \mbox{and}\quad\frac{\p u_2}{\p \nu_2}=\beta, 
$$ 
where $\nu_2=-\nu_1$.
 We will consider the Euclidean norm for the distance,  $d(x,y)=|x-y|$. Furthermore, we denote
\begin{equation}\label{J(ui)}J_r(u_i,x_0):=  \frac{1}{r^{2}} \int_{B_{r}(x_0)} \frac{\vert  \nabla u_i \vert^{2}}{\vert x-x_0 \vert^{n-2}} dx,
\end{equation}  and
\begin{equation}\label{J(u)}
J_r(u,x_0):=J_r(u_1,x_0)J_r(u_2,x_0).
\end{equation} 
When $x_0=0$ we simply write  $J_r(u)$ instead of $J_r(u,0)$.

\section{Main results}\label{mainresultssec}
\begin{thm} \label{existencethm} Let $\Om$ be a bounded  smooth domain of $\R^n$ and $f$
be a Lipschitz continuous function on $\partial \Om$ such that 
$f^+\not\equiv0$ and $f^-\not\equiv0.$ Then there exists a viscosity solution $u$ of \eqref{limitproblemk=2} such that $u_1=u^+\not\equiv0$ and
$u_2=u^-\not\equiv0$. Moreover $u$ is Lipschitz continuous in $\bar\Om$.
\end{thm}

\begin{thm}\label{limitpbthm}
Any  Lipschitz solution $u$ of \eqref{limitproblemk=2} such that $u_1=u^+\not\equiv0$ and
$u_2=u^-\not\equiv0$, satisfies in the viscosity sense the free boundary condition
$$\frac{\partial u_1}{\partial \nu_1}=\frac{\partial u_2}{\partial \nu_2}\quad\text{on }\Gamma=\partial\{u>0\}\cap\Om,$$
meaning that: if there exists a tangent ball $B$  at $x_0 \in\Gamma$,  such that either $B \subset \Om(u_1)$ or $B \subset \Om(u_2)$, then
there exists $\alpha>0$ such that 
 \begin{equation}\label{asymptodevumainthm}
 u(x)=\alpha <x-x_0,\nu_1>+o(|x-x_0|)
 \end{equation}
 where  $\nu_1$ is the normal vector to  $\partial B$ at $x_0$ pointing inward to $\Om(u_1)$ (and $\nu_2=-\nu_1$).
 In particular,  $u$ is a viscosity solution to the free boundary problem \eqref{freebddproblemintro}.
\end{thm}

\begin{thm}\label{C1alphaGammathm}
Let  $u$ be any Lipschitz solution of \eqref{limitproblemk=2} and 
let  $\mathcal{R}$ be  the set of regular points of $u$, according to Definition  \ref{firstdefinition}. Then $\mathcal{R}$  is an open subset  of $\Gamma$  and locally a surface of class 
 $C^{1,\alpha}$, with $0<\alpha\le1$.  In particular, $u$ has the asymptotic behavior \eqref{asymptodevumainthm} at any $x_0\in  \mathcal{R}$.
\end{thm}
The proofs of Theorems \ref{existencethm}, \ref{limitpbthm} and \ref{C1alphaGammathm}  are given respectively in   Sections  \ref{Existencesec}, \ref{Freeboundarycondisec} and  \ref{C1alphasec}.

\section{Backround: The monotonicity formula}\label{monotonicity}
In this section we recall the Alt-Caffarelli-Friedman monotonicity formula  and some related results that we will used later on in the paper. 
A proof  can be found in \cite{caffarelli_geometric_2005,petrosyan_regularity_2012}. We have: 

\begin{thm}[Monotonicity formula]\label{mootonicitythm}
Let $u_1,u_2\in C(B_1(0))$ be nonnegative subharmonic functions in $B_1(0)$. Assume $u_1u_2=0$ and   $u_1(0)u_2(0)=0$.  Let $u=u_1-u_2$ and 
$$J_r(u)=\frac{1}{r^4}\int_{B_r(0)} \frac{\vert  \nabla u_1 \vert^{2}}{\vert x\vert^{n-2}}\, dx \int_{B_r(0)} \frac{\vert  \nabla u_2 \vert^{2}}{\vert x \vert^{n-2}} \, dx,\quad 0<r<1.$$
Then $J_r(u)$ is finite and is a non-decreasing function of $r$. 
Moreover, $$J_r(u)\leq c(n)\|u_1\|^2_{L^2(B_1)}\|u_2\|^2_{L^2(B_1)},\quad 0<r\le\frac12.$$
\end{thm}

Theorem \ref{mootonicitythm} can be applied  to  $u=u_1-u_2$ solution  of \eqref{limitproblemk=2}. Indeed, by \eqref{puccifoprel} we have $\D u_1 \geq \puccin(u_1)= 0$ in $\Om(u_1)$  and $\Delta u_2\geq -\puccip(-u_2)=0$ in $\Om(u_2)$. Therefore, both $u_1$ and $u_2$ are subharmonic functions in the viscosity sense,  and thus is in the distributional sense, in the whole  $\Om$, 
\begin{equation}
\label{uno} \D u_1 \geq 0\quad \text{in } \Om
\end{equation} 
  and 
\begin{equation} 
\label{dos}  \Delta u_2\geq 0\quad \text{in } \Om.
\end{equation}

\begin{remark}
\label{remark: lower bound}
Since $J_r(u)$ is a monotone nonnegative function,  there exists
\begin{equation}
\label{zero0}
J_0(u):= \lim_{r\rightarrow 0^+} J_r(u).
\end{equation}
\end{remark}

The following theorem gives information on the case $J_r(u)$ constant. A proof of it can be found in \cite{petrosyan_regularity_2012}.

\begin{thm}
\label{blowupthm}
Let $u_1,u_2\in  C(B_1(0))$ be nonnegative subharmonic functions in $B_1(0)$. Assume $u_1u_2=0$ and  $u_1(0)u_2(0)=0$ and  let $u=u_1-u_2$. 
If $$J_{r_1}(u)=J_{r_2}(u)=:k$$ for some $0<r_1< r_2<1$,  
then, either one or the other of the following holds:
\begin{itemize}
\item[(i)] $u_1\equiv 0$ in $B_{r_2}(0)$ or $u_2\equiv 0$ in $B_{r_2}(0)$;
\item[(ii)] there exist a unit vector ${\bf \nu}$,  positive constants $\alpha, \beta$ and a universal positive constant  $c_n$, such that 
 $$k=c_n\alpha^2\beta^2$$ and 
  for any $x\in B_{r_2}(0)$,
$$
u(x) = \alpha <x, {\bf \nu}>^+ - \,\beta<x,{\bf \nu}>^-.
$$
\end{itemize}
\end{thm}


\section{Proof of Theorem \ref{existencethm}}\label{Existencesec}

We consider the Heaviside function $H:\real \rightarrow \{0,1\}$,
\begin{equation*}
 H (x)=
\begin{cases}
 & 1 \mbox{ when } x \geq 0\\
 & 0 \mbox{ when } x < 0,\\
\end{cases}
\end{equation*}
and we let $H_{\ep}$ denote a smooth approximation of $H$, satisfying $H'_{\ep} \geq 0$.
 Consider the  fully nonlinear uniformly elliptic operator $G$, defined by
\begin{equation}
\label{approx operator}
G(u):=H(u)\, \puccin(u)+(1-H(u))\,\puccip(u)
\end{equation} and its $\ep$-approximation $G_\ep,$ defined by
\begin{equation}
\label{approx operator}
G_\ep(u):=H_\ep(u)\, \puccin(u)+(1-H_\ep(u))\,\puccip(u).
\end{equation} 
 To prove existence of a Lipschitz solution of   \eqref{limitproblemk=2}, we  prove  that for any $\ep>0$,  there exists $u^\ep$  viscosity solution of the problem
\begin{equation}
\label{ep problem}
\begin{cases}
G_\ep(u^\ep)=0&\quad \mbox{in}\: \Om\\
u^\ep=f & \quad \mbox{on}\: \p\Om,\\
\end{cases}
\end{equation} 
and  that the  functions $u^\ep$'s  are  Lipschitz continuous uniformly in $\ep$.  
Existence of a solution of  \eqref{limitproblemk=2} will then follow by using the Ascoli-Arzel\`{a} Theorem and the stability of the viscosity solutions in the sets $\{u>0\}$ and $\{u<0\}$.
\begin{remark} \label{forcomparisson} By Proposition \ref{Fnproperties},  in the viscosity sense 
$$\ppuccin(u^\ep)\le\Fn(u^\ep)\leq G_\ep(u^\ep)\leq \Fp(u^\ep)\le \ppuccip(u^\ep)  $$
\end{remark}
We start by proving that any viscosity solution of  \eqref{ep problem} is Lipschitz continuous with Lipschitz norm independent of $\ep$. 

\begin{thm}\label{lipschthm} 
Let $\ep>0$,  $\Om$ be a bounded  smooth domain and  $G_\ep$  the operator defined in \eqref{approx operator}. Let  $f \in C^{0,1}(\p\Omega)$ satisfy
\begin{equation*}\|f\|_{ C^{0,1}(\p\Omega)}\leq K_0.\end{equation*}
Then, any  continuous viscosity solution $u^\ep$ of problem \eqref{ep problem} is Lipschitz continuous in $\overline{\Om}$ and  
\begin{equation*}\|u^\ep\|_{ C^{0,1}(\overline{\Om})}\leq C,\end{equation*}
where $C=C(n,\Om,\lambda,\Lambda, K_0)$.
\end{thm}

\begin{proof} 
Before giving the precise proof, we will give an heuristic argument, just to give an idea of the main technic. Assume  that  $u^\ep$ has a further regularity, for instance  $u^\ep \in C^{3}(\Om)$. Since   $\puccin(M)$ and $\puccip(M)$ are Lipschitz continuous with respect to $M\in \mathcal{A}_n$,   we have that $\puccin( u^\ep)$ and $\puccin(u^\ep)$  are Lipschitz continuous with respect to $x$,  therefore we can differentiate a.e. in $\Om$ both sides of the equation
$$
G_\ep(u^\ep )=0,
$$
in any  direction $\sigma \in \p B_1 (0)$. Indeed, if  we denote
$$
\mathcal{F}^\pm_{ij}(M):= \frac{\p \mathcal{F}^\pm }{\p m_{ij}}(M),
$$
where $M=(m_{ij})$, we obtain 
\begin{equation*}
\begin{split}
0=\p_\sigma G_\ep  (u^\ep)&= 
H_\ep(u^\ep)\puccin_{ij}(u^\ep) \p_{ij}(\p_\sigma u^\ep) + (1-H_\ep(u^\ep))\puccip_{ij}(u^\ep) \p_{ij}(\p_\sigma u^\ep)\\& + H^{'}_{\ep} (u^\ep)\left( \puccin (u^\ep) -\puccip (u^\ep) \right)\p_\sigma u^\ep.\\
\end{split}
\end{equation*}
Then, if $L$ denotes the linear operator, with coefficients that depend on $u^\ep$,  
$$L(\cdot):=\left(H_\ep(u^\ep) \,\puccin_{ij}(u^\ep)+(1-H_\ep(u^\ep))\,\puccip_{ij}(u^\ep)\right)\p_{ij}(\cdot),$$  we can see that $\p_\sigma u^\ep$ is a solution of 
$$
L(\p_\sigma u^\ep)+ \underbrace{H^{'}_{\ep} (u^\ep)\left( \puccin (u^\ep) -\puccip (u^\ep) \right)}_{\leq 0}\p_\sigma u^\ep=0.
$$
Since $H^{'}_{\ep} (u^\ep)\left( \puccin (u^\ep) -\puccip (u^\ep) \right)\leq 0$ by the maximum principle,
$$\sup_{\Om} \p_\sigma u^\ep \leq \sup_{\p \Om} \p_\sigma u^\ep.$$
Now, if $\sigma$ is a tangential direction to $\partial \Om$, then
$$\sup_{\p \Om} \p_\sigma u^\ep=\sup_{ \p\Om} \p_\sigma f\le \|\nabla f\|_{L^{\infty}(\partial\Om)}.$$
If $\sigma$ is a normal vector, then a barrier argument (as the one in Claim 1 below) shows that 
$$\sup_{\p \Om} \p_\sigma u^\ep\leq C.$$
Thus, $\p_\sigma u^\ep $ is bounded in $\Om$ and 
the arbitrarily of $\sigma$ implies the result. To overcome the lack of regularity, we will use standard techniques  from the theory of viscosity solutions.
In particular, we will discretize and  prove that the incremental quotient  is bounded, meaning that
there exists a constant $C_0$ independent of $\ep$ such that, for $\sigma \in \p B_1 (0)$ and $h>0$:
\begin{equation}
\label{goal}
\sup_{x,x+h\sigma\in\Om} \frac{u^\ep(x+\sigma h)-u^\ep(x)}{h}\leq C_0.
\end{equation}
Then the result holds true. To prove \eqref{goal}  we first prove the following claim:\\

\noindent{\em Claim 1 :} {\em There exists $C_0>0$ independent of $\ep$ such that  for any $x\in\overline{\Om}$ and  any $y\in\partial\Om$, 
$$|u^\ep(x)-u^\ep(y)|\leq C_0|x-y|.$$}
\noindent{\em Proof of Claim 1:} Consider the function $\psi$ solution to 
$$\begin{cases}
\ppuccin(\psi)=0&\text{in }\Om\\
\psi=f&\text{on }\partial\Om.\\
\end{cases}
$$ Then,  $\psi\in C^{2,\alpha}(\Om)$ and $\psi\in C^{0,1}(\overline{\Om})$, see \cite{caffarelli_cabre_1995,gilbarg_elliptic_2001}. 
  Remark \ref{forcomparisson} and comparison principle implies $u^\ep\geq \psi$  in $\overline{\Om}$.
  In particular, if $y\in\partial \Om$ and  $x \in \Om$, we have
\begin{equation}\label{claim1ulipbpoundary1}u^\ep(x)-u^\ep(y)\geq \psi(x)-\psi(y)\ge -C_0|x-y|,\end{equation} for some $C_0>0$ independent of $\ep$. 
Similarly, the inequality
\begin{equation}\label{claim1ulipbpoundary2}u^\ep(x)-u^\ep(y)\leq C_0|x-y|\end{equation}  follows by comparing $u^\ep$ with the solution $\varphi$  of 
$$\begin{cases}
\ppuccip(\varphi)=0&\text{in }\Om\\
\varphi=f&\text{on }\partial\Om.\\
\end{cases}$$
Claim 1 follows from estimates \eqref{claim1ulipbpoundary1} and \eqref{claim1ulipbpoundary2}. 

\medskip
Next, to prove  \eqref{goal}, assume by contradiction that, for some $\delta >0$,
\begin{equation*}
\sup_{x,x+h\sigma\in\Om}\frac{u^\ep(x+h\sigma )-u^\ep(x)}{h}\geq C_0+\delta,
\end{equation*} 
where  $C_0>0$ is given in Claim 1. 
Then, for $\kappa>0$, we have that 
\begin{equation}\label{contradiction}
\sup_{x,x+h\sigma\in\Om}\frac{u^\ep(x+h\sigma )+\kappa |x|^2-u^\ep(x)}{h}\geq C_0+\delta.
\end{equation} 
In what follows we will make explicit the dependence of the operator $G_\ep(u)$ in $u$ and $M\in \mathcal{S}_n$ by using the notation  $G_\ep(u,M)$. Denote $w^\ep_h(x):=u^\ep(x+h\sigma)+\kappa |x|^2$.
 Note that, by the uniformly ellipticity and the fact that $u^\ep$ is solution of  \eqref{ep problem}, 
  $w^\ep_h$ is a strict subsolution of $G(w^\ep_h-\kappa |x|^2,D^2w^\ep_h)=0$, in $\Om-h\sigma:=\{x-h\sigma\,|\,x\in\Om\}$ as it satisfies in the viscosity sense
 \begin{equation*}\begin{split}0&= G(w^\ep_h-\kappa |x|^2,D^2w^\ep_h-2\kappa I_n)\leq G(w^\ep_h-\kappa |x|^2,D^2w^\ep_h)-\ppuccin(2\kappa I_n)
 \\&=G(w^\ep_h-\kappa |x|^2,D^2w^\ep_h)-2\kappa\lambda n,
 \end{split}\end{equation*}
  from which
 \begin{equation}\label{strictsubsolution}G_\ep(w^\ep_h-\kappa |x|^2,D^2w^\ep_h)\ge  2\kappa\lambda n\quad\text{in } \Om-h\sigma.
 \end{equation} 
In order to infer a differential  inequality  satisfied by $w^\ep_h(x)-u^\ep(x)$ in the viscosity sense, consider
for any fixed $\tau_0>0$ and $0<\tau<\tau_0$,  the upper $\tau$-envelope of $w^\ep_h$ and the lower $\tau$-envelope of $u^\ep$ defined respectively by
$$
w^\tau(x):=\sup_{y\in\Om_{\tau_0}-h\sigma}  \left\{ w_h^\ep(y)+\tau -\frac{1}{\tau}|y-x|^2\right\},\quad x\in\Om_{\tau_0}-h\sigma
$$
$$
u_\tau(x):=\inf_{y\in \Om_{\tau_0}}  \left\{ u^\ep(y)-\tau +\frac{1}{\tau}|y-x|^2\right\},  \quad x\in\Om_{\tau_0}, 
$$
where $\Om_{\tau_0}:=\{x\in\Om\,|\,d(x,\partial\Om)> \tau_0\}$ and, for simplicity of notation, we have dropped the dependence on $h$ and $\ep$. 
\medskip

\noindent{\em Claim 2 :} The  upper and lower $\tau$-envelopes have the following properties:
\begin{itemize}
\item[a)]  $w^\tau\in C(\Om_{\tau_0}-h\sigma),\,u_\tau\in C(\Om_{\tau_0})$, $w^\tau\ge  w^\ep_h+\tau$, $u_\tau\le u^\ep-\tau$, $w^\tau\to w^\ep_h$ and $u_\tau\to u^\ep$ as $\tau\to 0$ uniformly  in 
$\Om_{\tau_0}-h\sigma$ and in $\Om_{\tau_0}$ respectively.
\item[b)] For any  $x\in \Om_{\tau_0}\cap (\Om_{\tau_0}-h\sigma)$ there exists a concave  (resp., convex) paraboloid of opening $2/\tau$ that touches $w^\tau$ (resp., $u_\tau$) by below (resp., above) at $x$.
In particular, $w^\tau$ and $u_\tau$ are punctually second order differentiable almost everywhere in  $ \Om_{\tau_0}\cap (\Om_{\tau_0}-h\sigma)$, meaning that, for a.e. $x_0 \in \Omega_{\tau_{0}}\cap (\Om_{\tau_0}-h\sigma)$ there exist paraboloids $P_w$ and $P_u$,  such that, $w^\tau(x)= P_w(x) +o(|x-x_0|^2)$ and $u_\tau(x)= P_u(x) +o(|x-x_0|^2)$ as $x \rightarrow x_0$. 
\item[c)] If  $x^\tau\in\Om_{\tau_0}-h\sigma$is such that $w^\tau(x)=w_h^\ep(x^\tau)+\tau-\frac{1}{\tau}|x-x^\tau|^2$, then 
$$\frac{1}{\tau}|x-x^\tau|^2\leq w_h^\ep(x^\tau)-w_h^\ep(x).$$
If $x_\tau\in\Om_{\tau_0}$ is such that $u_\tau(x)=u^\ep(x_\tau)-\tau+\frac{1}{\tau}|x-x_\tau|^2$, then 
$$\frac{1}{\tau}|x-x_\tau|^2\leq u^\ep(x)-u^\ep(x_\tau).$$
\item[d)] There exists $\tau_1>0$ such that for any $\tau<\tau_1$, $w^\tau$ is  a viscosity (and therefore a.e.) subsolution to 
$$G_\ep\left(w^\tau(x)-\tau+\frac{1}{\tau}|x-x^\tau|^2-\kappa|x^\tau|^2,D^2 w^\tau(x)\right)=2\kappa\lambda n,\quad x\in\Om_{2\tau_0}-h\sigma$$ and 
$u_\tau$ is a viscosity (and therefore a.e.)  supersolution to 
$$G_\ep\left(u_\tau(x)+\tau-\frac{1}{\tau}|x-x_\tau|^2,D^2 u_\tau(x)\right)=0,\quad x\in\Om_{2\tau_0}.$$
\end{itemize}
\noindent{\em Proof of Claim 2:} For the proofs of (a)-(c) see Theorem  5.1 and Lemma 5.2 in  \cite{caffarelli_cabre_1995}. 
 Note that by (c), 
$$\frac{1}{\tau}|x-x^\tau|^2\leq w_h^\ep(x_\tau)-w_h^\ep(x)\leq C.$$ 
Since $w_h^\ep$ is continuous this implies that
\begin{equation}\label{x^tautox}\frac{1}{\tau}|x-x^\tau|^2\to0\quad\text{as }\tau\to0.
\end{equation}
Similarly,
\begin{equation}\label{x_tautox}\frac{1}{\tau}|x-x_\tau|^2\to0\quad\text{as }\tau\to0.
\end{equation}
To prove (d), let $x_0\in\Om_{2\tau_0}-h\sigma$ and let $P(x)$ be a paraboloid that touches by above $w^\tau$ at $x_0$. Consider the paraboloid
$$Q(x)=P(x+x_0-x_0^\tau)+\frac{1}{\tau}|x_0-x_0^\tau|^2-\tau.$$
By \eqref{x^tautox}, we can pick $\tau_1>0$ independent of $x_0$ such that $x_0^\tau\in\Om_{\tau_0}-h\sigma$ for any  $\tau<\tau_1$. Take any $x$ sufficiently close to $x_0^\tau$, so that  
$x+x_0-x_0^\tau\in \Om_{\tau_0}-h\sigma$, then, by definition of $w^\tau$, 
$$w_h^\ep(x)\leq w^\tau(x+x_0-x_0^\tau)+\frac{1}{\tau}|x_0-x_0^\tau|^2-\tau\leq Q(x).$$
Moreover, $w_h^\ep(x_0^\tau)=Q(x_0^\tau)$, since $w^\tau(x_0)=P(x_0).$ Hence $Q$ touches $w_h^\ep$ by above at $x_0^\tau$ and by \eqref{strictsubsolution},
$$2\kappa\lambda n\leq G_\ep(w_h^\ep(x_0^\tau)-\kappa|x_0^\tau|^2,D^2 Q(x_0^\tau))=G_\ep\left(w^\tau(x_0)-\tau+\frac{1}{\tau}|x_0-x_0^\tau|^2-\kappa|x_0^\tau|^2,D^2 P(x_0)\right).$$
Similarly one can prove the second viscosity inequality in (d). This concludes the proof of Claim 2.

\medskip
Let us continue the proof of the theorem.  We have assumed that \eqref{contradiction} is true. If the supremum in \eqref{contradiction} is attained at $\bar{x}$, 
then both $\bar{x}$ and $\bar{x}+\sigma h$ have to be in the interior  of $\Omega$, for otherwise we would have  $u^{\ep}(\bar{x}+ h \sigma ) - u^{\ep}(\bar{x} )\geq C_0 h + \delta h -\kappa |\bar{x}|^2 \geq C_0 h + \delta h -\kappa C(\Omega) $ which contradicts Claim 1 for $\kappa < \delta h/ C(\Omega)$. Thus, there exists
$\tau_0>0$ such that
$$\sup_{x+h\sigma,x\in\Om}\frac{w_{h}^\ep(x)-u^\ep(x)}{h}=\sup_{x+h\sigma,x\in\Om_{3\tau_0}}\frac{w_{h}^\ep(x)-u^\ep(x)}{h}\geq C_0+\delta.$$ 
For $\tau$ small enough, by (a) of Claim 2, there exists $x_0\in \Om_{\frac{5\tau_0}{2}}\cap\big(\Om_{\frac{5\tau_0}{2}}-h\sigma\big)$ such that  \begin{equation}\label{maxequa}\sup_{x+h\sigma,x\in\Om_{\tau_0}}\frac{w^\tau(x)-u_\tau(x)}{h}=\frac{w^\tau(x_0)-u_\tau(x_0)}{h}=M\geq C_0.\end{equation}
Take $s>0$ and  $r<\tau_0/2$ small enough so that $B_r(x_0)\subset \Om_{2\tau_0}\cap\big(\Om_{2\tau_0}-h\sigma\big)$ and define
$$v(x):=u_\tau(x)-w^\tau(x)+Mh+s|x-x_0|^2-sr^2.$$
Then  $v$ has a strict minimum at $x_0$, moreover
\begin{equation}\label{toapplyconvexenv}v(x)\geq 0\quad\text{on }\partial B_r(x_0)\quad\text{and}\quad v(x_0)=-sr^2<0.\end{equation}
Let us denote by $\Gamma_v$ the convex envelope of $-v^-$ in $B_{2r}(x_0)$, where we have extended $v\equiv 0$ outside $B_r(x_0)$. 
 Here we use standard techniques from the theory of viscosity solutions, see \cite{caffarelli_cabre_1995}. Since we do not know if
$w^\tau$ and $u_\tau$ are twice differentiable at $x_0$, we introduce the convex envelope in order to find a point $x_1$ of twice differentiability for both 
$w^\tau$ and $u_\tau$ such that $w^\tau(x_1)>u_\tau(x_1)$ and $D^2w^\tau(x_1)\leq D^2u_\tau(x_1)$+small corrections.
We have that $\Gamma_v\leq 0$ in  $B_{2r}(x_0)$. 
By (b) of Claim 2, for any $x\in\overline{B_r(x_0)}\cap\{v=\Gamma_v\}$ there exists a convex paraboloid
with opening independent of $x$ that touches $\Gamma_v$ by above. By Lemma 3.5 in  \cite{caffarelli_cabre_1995}, $\Gamma_v\in C^{1,1}(B_r(x_0))$ and 
$$\int_{B_r(x_0)\cap\{v=\Gamma_v\}}\text{det}D^2\Gamma_v\,dx>0.$$
In particular $|B_r(x_0)\cap\{v=\Gamma_v\}|>0$. 
Since $w^\tau$ and $u_\tau$ are second order differentiable almost everywhere in $B_r(x_0)$, there exists $x_1\in B_r(x_0)\cap\{v=\Gamma_v\}$ such that $w^\tau$ and $u_\tau$ are second order differentiable at $x_1$ and by  (d) of Claim 2,
\begin{equation}\label{GWtau}G_\ep\left(w^\tau(x_1)-\tau+\frac{1}{\tau}|x_1-x_1^\tau|^2-\kappa| x_1^\tau|^2, D^2w^\tau(x_1)\right)\ge 2\kappa\lambda n\end{equation} and 
\begin{equation}\label{GUtau}G_\ep\left(u_\tau(x_1)+\tau-\frac{1}{\tau}|x_1-(x_1)_\tau|^2,D^2 u_\tau(x_1)\right)\le0.\end{equation} 
Since $\Gamma_v$ is convex, $\Gamma_v\leq v$ and $\Gamma_v(x_1)=v(x_1)$, we have that $D^2v(x_1)\geq 0$, i.e., 
\begin{equation}\label{negativematrixmax}D^2 w^\tau(x_1)\leq D^2 u_\tau(x_1)+2sI_n.\end{equation}
 Moreover, since $\Gamma_v$ is negative in $B_{r}(x_0)$, we have that
\begin{equation*}
w^\tau(x_1)>u_\tau(x_1)+s|x_1-x_0|^2-sr^2+Mh.\end{equation*}
 In particular, for $s$ and $r$ small enough
\begin{equation}\label{wx1vx1}w^\tau(x_1)>u_\tau(x_1)+\frac{Mh}{2}.\end{equation}
Let us denote $\varphi^\tau(x_1):=-\tau+\frac{1}{\tau}|x_1-x_1^\tau|^2$ and 
$\varphi_\tau(x_1):=\tau-\frac{1}{\tau}|x_1-(x_1)_\tau|^2$.
Then, by subtracting the  inequalities \eqref{GWtau} and  \eqref{GUtau}, we get
\begin{equation*}\begin{split}
2\kappa\lambda n&\leq G_\ep\left(w^\tau(x_1)+\varphi^\tau(x_1)-\kappa| x_1^\tau|^2,D^2 w^\tau(x_1)\right)-G_\ep\left(u_\tau(x_1)+\varphi_\tau(x_1),D^2 u_\tau(x_1)\right)\\&
=H_\ep\left(w^\tau(x_1)+\varphi^\tau(x_1)-\kappa| x_1^\tau|^2\right)\puccin( w^\tau)(x_1)\\&+\left(1-H_\ep\left(w^\tau(x_1)+\varphi^\tau(x_1)-\kappa| x_1^\tau|^2\right)\right)\puccip( w^\tau)(x_1)\\&
-H_\ep\left(u_\tau(x_1)+\varphi_\tau(x_1)\right)\puccin(u_\tau)(x_1)-(1-H_\ep\left(u_\tau(x_1)+\varphi_\tau(x_1)\right))\puccip(u_\tau)(x_1).
\end{split}
\end{equation*}
Adding and subtracting 
$$H_\ep\left(u_\tau(x_1)+\varphi_\tau(x_1)\right)\puccin( w^\tau)(x_1)$$ and 
$$[1-H_\ep\left(u_\tau(x_1)+\varphi_\tau(x_1)\right)]\puccip( w^\tau)(x_1),$$
in the right hand-side of the inequality above, we obtain

\begin{equation*}\begin{split}
2\kappa\lambda n&
\leq [H_\ep\left(w^\tau(x_1)+\varphi^\tau(x_1)-\kappa| x_1^\tau|^2\right)-H_\ep\left(u_\tau(x_1)+\varphi_\tau(x_1)\right)]\puccin( w^\tau)(x_1)\\&
+H_\ep\left(u_\tau(x_1)+\varphi_\tau(x_1)\right)\puccin( w^\tau)(x_1)\\&
-[H_\ep\left(w^\tau(x_1)+\varphi^\tau(x_1)-\kappa| x_1^\tau|^2\right)-H_\ep\left(u_\tau(x_1)+\varphi_\tau(x_1)\right)]\puccip( w^\tau)(x_1)\\&
+[1-H_\ep\left(u_\tau(x_1)+\varphi_\tau(x_1)\right)]\puccip( w^\tau)(x_1)\\&
-H_\ep\left(u_\tau(x_1)+\varphi_\tau(x_1)\right)\puccin(u_\tau)(x_1)-(1-H_\ep\left(u_\tau(x_1)+\varphi_\tau(x_1)\right))\puccip(u_\tau)(x_1)\\&
=[H_\ep\left(w^\tau(x_1)+\varphi^\tau(x_1)-\kappa| x_1^\tau|^2\right)-H_\ep\left(u_\tau(x_1)+\varphi_\tau(x_1)\right)]\puccin( w^\tau)(x_1)\\&
-[H_\ep\left(w^\tau(x_1)+\varphi^\tau(x_1)-\kappa| x_1^\tau|^2\right)-H_\ep\left(u_\tau(x_1)+\varphi_\tau(x_1)\right)]\puccip( w^\tau)(x_1)\\&
+H_\ep\left(u_\tau(x_1)+\varphi_\tau(x_1)\right)[\puccin( w^\tau)(x_1)-\puccin( u_\tau)(x_1)]\\&
+[1-H_\ep\left(u_\tau(x_1)+\varphi_\tau(x_1)\right)][\puccip( w^\tau)(x_1)-\puccip( u_\tau)(x_1)]\\&
\leq[H_\ep\left(w^\tau(x_1)+\varphi^\tau(x_1)-\kappa| x_1^\tau|^2\right)-H_\ep\left(u_\tau(x_1)+\varphi_\tau(x_1)\right)][\puccin( w^\tau)(x_1)-\puccip( w^\tau)(x_1)]\\&
+\mathcal{M}^+( w^\tau- u_\tau)(x_1).
\end{split}
\end{equation*}
We have obtained
\begin{equation}\label{chianineqfres}\begin{split}
2\kappa\lambda n&\leq
[H_\ep\left(w^\tau(x_1)+\varphi^\tau(x_1)-\kappa| x_1^\tau|^2\right)-H_\ep\left(u_\tau(x_1)+\varphi_\tau(x_1)\right)][\puccin( w^\tau)(x_1)-\puccip( w^\tau)(x_1)]\\&
+\mathcal{M}^+( w^\tau- u_\tau)(x_1)
\end{split}
\end{equation}
Now, by \eqref{x^tautox} and \eqref{x_tautox} we have that $\varphi^\tau(x_1),\varphi_\tau(x_1)\to 0$ as $\tau\to 0$. This,  combined with \eqref{wx1vx1}, yields
$$w^\tau(x_1)+\varphi^\tau(x_1)-\kappa| x_1^\tau|^2>u_\tau(x_1)+\varphi_\tau(x_1)$$ 
 for $s, \kappa, \tau$ small enough.  Since $H_\ep$ is non-decreasing and $\puccin( w^\tau)(x_1)-\puccip( w^\tau)(x_1)\leq 0$, we infer that
$$[H_\ep\left(w^\tau(x_1)+\varphi^\tau(x_1)-\kappa| x_1^\tau|^2\right)-H_\ep\left(u_\tau(x_1)+\varphi_\tau(x_1)\right)][\puccin( w^\tau)(x_1)-\puccip( w^\tau)(x_1)]\leq 0.$$
Next, from $\eqref{negativematrixmax}$,
$$\mathcal{M}^+( w^\tau- u_\tau)(x_1)\leq 2sn\Lambda.$$
Plugging the last two inequalities into \eqref{chianineqfres}, we obtain
$$2\kappa\lambda n\leq 2sn\Lambda,$$ which is a contradiction for $s$ small enough ($s<\kappa\lambda/\Lambda$).

We have proven that for any $\delta>0$,

$$\sup_{x,x+h\sigma\in\Om}\frac{u^\ep(x+h\sigma )-u^\ep(x)}{h}\leq C_0+\delta.$$
Letting $\delta$ go to 0, we get \eqref{goal}. 

Note that comparing $u^\ep$ with the sub and supersolution introduced in Claim 1, we infer that there exists $C_1>0$ independent of $\ep$ such that 
$$\|u^\ep\|_{L^\infty(\Om)}\le C_1.$$
This bound combined with \eqref{goal} yields a uniform in $\ep$ estimate of the Lipschitz norm of the solution $u^\ep$ of  \eqref{ep problem}.
Thus the theorem is proven.
\end{proof}

\begin{thm}
\label{existence}
Under the assumptions of Theorem \ref{lipschthm},
there exists a continuous  viscosity solution $u^\ep$ of the $\ep$-problem \eqref{ep problem}. 
Moreover, 
\begin{equation}\label{puccinleqpuccis}\Fn(u^\ep)\leq0\leq\Fp(u^\ep)\end{equation} in the viscosity sense in $\Om.$
\end{thm}

\begin{proof}
We fix $\ep> 0$. For $\alpha\in(1/2,1)$, let  $\Theta:=C^{0,\alpha}(\overline{\Om})$ be the Banach space of $\alpha$-\holder continuous functions defined on 
$\overline{\Omega}$. 
Let $T$ be the operator  defined in the following way, for $u\in\Theta$,
$$T (u )= v $$
if 
$v$ is  the  viscosity solution of 
\begin{equation}
\begin{cases}
\label{T}
H_\ep(u)\, \puccin (v)+(1-H_\ep(u))\,\puccip (v) =0 &\quad \mbox{in} \:\Omega   \\
v = f &\quad \mbox{on} \:\p\Omega. 
\end{cases}
\end{equation}
 Note  that $T$ is well defined. Indeed by Proposition \ref{Fnproperties} the operator $G_{\color{red} \ep}(x,v):=H_\ep(u)\, \puccin (v)+(1-H_\ep(u))\,\puccip (v)$ is uniformly elliptic.
 Moreover, since $u\in C^{0,\alpha}(\overline{\Om})$ with $\alpha>1/2$, $G_{\color{red} \ep}(x,v)$ satisfies the comparison principle, see  \cite[Theorem III.1]{ishiilions}.
 Let $\psi$ and $\varphi$ be the   solutions of 
  \begin{equation*}
\begin{cases}
\ppuccin(\psi) =0 &\quad \mbox{in} \:\Omega   \\
\psi = f &\quad \mbox{on} \:\p\Omega 
\end{cases}
\quad
\text{and}\quad
\begin{cases}
\ppuccip(\varphi) =0 &\quad \mbox{in} \:\Omega   \\
\varphi = f &\quad \mbox{on} \:\p\Omega. 
\end{cases}
\end{equation*}
Then, $\psi$ and $\varphi$ are respectively  sub and supersolution of \eqref{T}.  Thus, by the Perron's method, there exists a unique viscosity solution of \eqref{T}.

 Observe that if $T $ has a fixed point $u^{\ep}$, that is  $T (u^{\ep})=u^{\ep} $, then     
$u^\ep$ is solution to  \eqref{ep problem}.
Moreover, by Remark \ref{forcomparisson}, \eqref{puccinleqpuccis} also follows.
We will prove that we can apply the Leray-Schauder fixed point  theorem  \cite[Theorem 11.3]{gilbarg_elliptic_2001}  and conclude that $T$ has a fixed point, which concludes the proof.
We have:
\begin{enumerate}
\item ${\it{ T(\Theta) \subset \Theta:}} $ 
Let $v=T(u)$, then regularity theory   implies that $v \in C^{0,\beta}(\overline{\Omega})$, for any $\beta\in (0,1)$, see   \cite[Theorem VII.1]{ishiilions}. 
In particular $v\in\Theta$.

\item {\it $T$ is continuous:} Let  $\{u_n \} \subset \Theta$ be  such
that $u _n\rightarrow \bar u$ in $\Theta$.
We need to prove that $v_n:=T(u_n)\to \bar v:=T( \bar u)$ in  $\Theta$.
By the \holder estimates for the solutions $v_n$,  \cite[Theorem VII.1]{ishiilions}, 
 we have that
$\|v_n\|_{C^{0,\beta}(\overline{\Om})}\leq C$ for $\beta>\alpha$. Since  the subset of $\Theta$  of $\beta$-H\"older continuous functions on $\overline{ \Om}$ is precompact in $\Theta$, we can extract from $\{v_n\}$ a convergent subsequence.
Let $\{v_{n_k}\}$ be any convergent subsequence, $v_{n_k}\to w$ as $k\to+\infty$  in $\Theta$, then by  the stability of viscosity solutions  under uniform convergence, it follows that 
$w$ solves \eqref{T} with $u=\bar u$ , that is $w=T(\bar u)=\bar v$. Since every convergent subsequence converges to the same limit function $\bar v$, we have that the full sequence 
$\{v_n\}$  converges to $\bar v$ in $\Theta.$

\item {\it $T $ is compact:} By  the H\"older estimates, T maps bounded set of $\Theta$ into bounded sets of $C^{0,\beta}(\overline{\Om})$, $\beta>\alpha$ which are precompact in $\Theta$.
\item {\it There exists $M>0$ such that $\|u\|_\Theta<M$ for all $u\in\Theta$ and $\sigma\in[0,1]$ satisfying $u=\sigma T(u)$:
the equation $u=\sigma T(u)$ is equivalent to the Dirichlet problem 
\begin{equation*}
\begin{cases}
H_\ep(u)\, \puccin (u)+(1-H_\ep(u))\,\puccip (u) =0 &\quad \mbox{in} \:\Omega   \\
u = \sigma f &\quad \mbox{on} \:\p\Omega. 
\end{cases}
\end{equation*}
The estimate $\|u\|_\Theta<M$, for some $M>0$, then follows from Theorem \ref{lipschthm}.
}
\end{enumerate}

This concludes the proof of the existence of a fixed point $u^\ep$ and thus of  a solution of \eqref{ep problem} satisfying  \eqref{puccinleqpuccis}.

\end{proof}

\subsection{Proof of Theorem \ref{existencethm}}
By Theorem \ref{existence}, for any $\ep>0$ there exists $u^\ep$   viscosity solution of  \eqref{ep problem}, satisfying also \eqref{puccinleqpuccis}.
By Theorem \ref{lipschthm}  the sequence $\{u^\ep\}$ is uniformly Lipschitz continuous, thus by the Ascoli-Arzel\`{a} Theorem there exists a subsequence of 
 $\{u^\ep\}$ uniformly convergent to a Lipschitz function $u$ solution to \eqref{limitproblemk=2}.
 
 If $f^+,\,f^-\not\equiv0$, then by the Lipschitz regularity of $u$ up to the boundary of $\Omega$, we have $u_1=u^+\not\equiv0$ and  $u_2=u^-\not\equiv0$.


\section{Non-degeneracy at regular points}\label{regulapointssec}

In this section we introduce the definition of regular points for $u$ solution of \eqref{limitproblemk=2}. These are points where at least one among $u_1$ and $u_2$ has linear growth away from the free boundary, where here and  throughout this section we will use the notation introduced in \eqref{pos neg}.

\begin{defn}
\label{firstdefinition}
Let $u$ be a solution of problem \eqref{limitproblemk=2}. Consider $x_0 $ a point on the free boundary $\Gamma$. 
We say that $x_0$ is a {\em regular point} if there exist positive constants $\tilde{r}=\tilde{r} (x_0)$  and $M=M(x_0)$ such that 
\begin{equation}\label{regularpointconfdef}
\sup_{B_{r}(x_0)}U\geq Mr,
\end{equation}
for every $0<r < \tilde{r}$, where 
$$U(x):=\max\{u_1(x),u_2(x)\}=|u(x)|.$$
 Otherwise, we say that $x_0$ is a {\em singular point}. 
\end{defn}


\begin{lem}\label{onesidepointsregular} 
 Let $u$ be a solution of problem \eqref{limitproblemk=2}. If $\Ga$ has a ball from one side at $x_0\in \Ga$, that is there exists a ball $B_{r_0}(y)$ contained inside  the support of either $u_1$ or $u_2$, such that $x_0\in\partial  B_{r_0}(y)$, then $x_0$ is a regular point. 
\end{lem}
\begin{proof}
Indeed, suppose, without loss of generality,  that $B_{r_0}(y)\subset \Omega(u_1)$. By \eqref{limitproblemk=2} $u_1$ is solution of $\puccin (u_1)= 0$  in $B_{r_0}(y)$. Then by the Harnack inequality we have that $u_1(x)\geq M_1$, for any $x\in  \overline{B}_{\frac{r_0}{2}}(y)$ with  $M_1=Cu_1(y)$, where $C$ is a universal constant. Let $\phi$ be the solution of 
$$\begin{cases}
\ppuccin (\phi) = 0&\text{in }B_{r_0}(y)\setminus B_{\frac{r_0}{2}}(y)\\
\phi=M_1&\text{on }\partial B_{\frac{r_0}{2}}(y)\\
\phi=0&\text{in } \partial B_{r_0}(y),
\end{cases}
$$ that is,   $\phi(x)= M_1 \frac{1}{2^\gamma-1}\left(\frac{r_0^{\gamma}}{|x-y|^\gamma}- 1\right)$, where $\gamma = \Lambda(n-1)/\lambda-1$ and $\lambda$ and $\Lambda$ are the elliptic constants of the Pucci's operator $\puccin $ (see Lemma \ref{fondamentalsollem} in Appendix).
Then, since in  $B_{r_0}(y)\setminus B_{\frac{r_0}{2}}(y)$ 
$$ \puccin(\phi) \ge \ppuccin (\phi)=0= \puccin (u_1)$$ 
and $u_1 \geq \phi$ on $\partial B_{r_0}(y) \cup \partial B_{r_0/2}(y)$,  the comparison principle  and  (iii) of Lemma \ref{fondamentalsollem} imply that  for any $x\in B_{r_0}(y)\setminus B_{\frac{r_0}{2}}(y)$,
$$u_1(x)\ge \phi(x)\geq  \frac{M_1 \gamma}{r_0 (2^{\gamma}-1)} d(x,\partial B_{r_0} (y)).$$
Hence,   for any $r < \frac{r_0}{2}$,
$$\sup_{B_r(x_0)}u_1\geq  \frac{M_1 \gamma}{r_0 (2^{\gamma}-1)} r=:M r.$$
Therefore, \eqref{regularpointconfdef} holds with $\tilde{r}(x_0)=\frac{r_0}{2}$ and 
$$M=\frac{u_1(y) C\gamma}{r_0(2^{\gamma}-1)}$$ depending only on $x_0$, $n$, $\lambda$ and $\Lambda$. 
%
\end{proof}

\begin{remark} The set of regular points is dense in $\Gamma$. Indeed, since $\Omega (u_1)$ is an open set, the set of points  in $\Gamma$ with an interior tangent ball is dense in  $\Gamma$. To see it,  let $x$ be any point  in $\Gamma$. Let us consider a sequence of points $\{x_k\}$ contained in 
$\Omega (u_1)$ and converging to $x$ as $k\to \infty$. Let $d_k$ be the distance of $x_k$ from  $\Gamma$. Then the balls $B_{d_k}(x_k)$ are contained in 
$\Omega (u_1)$  and there exist points  $y_k \in \Gamma \cap \partial B_{d_k}(x_k)$ where the $x_k$'s realize the distance from  $\Gamma$. The sequence $\{y_k\} \subset \Gamma$ is a sequence of  points  that have a tangent ball from the inside  and converges to $x$.
\end{remark}


 The following lemma states that at regular points both functions $u_1$ and $u_2$ have linear growth away from the free boundary.
\begin{lem}
\label{lemma: both bounds}
Let $u$ be a Lipschitz solution of problem \eqref{limitproblemk=2}  and  let $z\in\Ga$ be  a regular point. Then, there exist $c=c(z)$ and $C$ positive constants such that, for any $ 0< r < \tilde{r}(z)$,
 $$
 c\, r \leq \sup_{B_r (z)} u_i  \leq C\, r \qquad {i=1,2.}
$$
\end{lem}
\begin{proof}

The inequality 
$
\sup_{B_r (z)} u_i \leq C\, r
$
for $i=1,2$ follows from the Lipschitz regularity of $u$.
We prove that if \eqref{regularpointconfdef} holds true for $x_0=z$, then 
\begin{equation}
\label{ndg}
\sup_{B_r (z)} u_i(x) \geq c\, r\quad\text{for any }0<r< \tilde{r}(z)<d(z,\partial\Om),\quad i=1,2,
\end{equation} for some $c=c(M)$.
Assume by contradiction that for    $\ep < \frac{M}{4}$  there exists  $0<\rho < \tilde{r}$ such that, w.l.o.g. 
\begin{equation}
\label{cont ndg}
\sup_{B_\rho (z)} u_2< \ep\, \rho.
\end{equation}
Set $r_\rho: =\frac{\rho}{4}< \tilde{r}$. From \eqref{cont ndg} we have that  $\sup_{B_{r_\rho} (z)} u_2 \leq \sup_{B_{\rho}(z)} u_2 < \ep \rho$ and hence 
\begin{equation}
\label{angel}
\sup_{B_{r_\rho}(z)} u_2 < (4 \ep) \frac{\rho}{4} < M \frac{\rho}{4}=Mr_\rho. 
\end{equation}
 Therefore from \eqref{angel} and the fact that $\sup_{B_{r_\rho} (z)} U \geq M r_\rho$, where $U=\max\{u_1,u_2\}$, we must have
$$\sup_{B_{r_\rho}(z)} u_1 \geq  Mr_\rho;$$ that is, there exists $y \in \overline{B}_{r_\rho} (z)$ such that 
$$
u_1(y)\geq Mr_\rho.
$$ 
Consider the ball centered at $y$ with  radius $h$,  where $h=\abs{y-x_0}$, being $x_0 \in \Ga $ the closest point to $y$ in $\Gamma$. 
Observe that $h \leq \frac{\rho}{4}$. 
Next, the ball $B_{h}(y)$ is contained in $\Om(u_1)$, therefore  by \eqref{limitproblemk=2} 
$\puccin (u_1)= 0$ in $B_{h}(y)$. The Harnack inequality then implies $u_1\geq CM r_{\rho}$ on $\bar{B}_{\frac{h}{2}}(y)$, 
where $C$ is a universal constant. 
Let  $\phi$ be the function defined as follows:
\begin{equation}
\label{barrier}
\phi(x) =  \ds CM\ds\frac{ r_{\rho}}{2^\gamma -1}\left(\ds\frac{{h}^{\gamma}}{\abs{x-y}^{\gamma}}{}-1\right),
\end{equation}
with $\gamma = \frac{\Lambda (n-1) -\lambda}{\lambda}$.  Then, $\phi$ satisfies
\begin{equation*}
\begin{cases}
\ppuccin \left( \phi \right) =0 &  \text{in }\R^n\setminus \{y\}\\
  \phi  = CM \,r_{\rho} & \text{ on } \p B_{\frac{h}{2}} (y)\\
  \phi= 0  & \text{ on }\p B_{h}(y),\\
\end{cases}
\end{equation*} 
see Lemma \ref{fondamentalsollem}. In particular, 
since $u_1\geq \phi$ on $\p B_{h}(y) \cup \p B_{\frac{h}{2}} (y)$, 
 by the comparison principle, $$u_1(x)\geq \phi (x), \qquad x \in    B_{h}(y) \backslash B_{\frac{h}{2}} (y).$$
 The previous inequality still holds in the complement of $B_{h} (y)$ in $\Om(u_1)$, being $\phi$ negative in that set. Therefore, we have that 
 \begin{eqnarray}\label{bound u_1bis}
u(x)= u_1(x)\geq \phi (x)\quad \text{ if }x \in \Om(u_1)\setminus B_{\frac{h}{2}} (y).  
 \end{eqnarray}
To continue the proof, we will 
  prove that $\phi\leq -u_2$ in a neighborhood of $x_0$ in $\Om(u_2)$. If $x\in B_{2h}(y)$ then $d(x,z)\leq d(x,y)+d(y,z)\leq 2h+r_\rho\leq \frac{3}{4}\rho,$ therefore $B_{2h}(y)\subset B_\rho(z)$. In particular, 
 by \eqref{cont ndg}, 
$$
\sup_{B_{2h} (y)} u_2 <  \ep\rho=\ep4r_\rho.
$$
On the other hand, 
$$\phi=-\frac{CM r_{\rho}}{2^\gamma}  \text{ on  }  \p B_{2h}(y)\quad\text{and}\quad  
\phi\leq 0\text{ on }\Gamma.
$$ 
Let $\ep$ be so small that $4\ep\le\frac{CM}{ 2^\gamma}$, then
 $$ \phi\leq -u_2\quad\text{on }\partial (\Om(u_2)\cap B_{2h}(y))\quad\text{ and }\quad\phi<-u_2 \quad\text{on }\Gamma \cap B_{2h}(y).$$
Since in addition,  in the set $\Om(u_2)\cap B_{2h}(y)$ we have
$$\Fn(\phi)\ge \ppuccin \left( \phi \right)= 0\quad\text{and}\quad\puccin(-u_2)=\puccin(u)\leq0,$$
the strong maximum  principle implies 
$$\phi(x)<-u_2(x)\quad\text{for any }x \in \Om(u_2)\cap B_{2h}(y).$$
 Putting all together, by \eqref{bound u_1bis} and the previous inequality, 
  we conclude that for all $x \in B_\frac{h}{2}(x_0) $ the function $u=u_1-u_2$ satisfies 
  \begin{equation}
  \label{bound diff}
u(x)\geq \phi(x),\,u\not\equiv\phi \mbox{  and  }  u(x_0)=\phi(x_0).
  \end{equation}
 This  is in contradiction with  the strong maximum principle, since we know that 
  $\puccin(u)\leq 0\le \Fn(\phi)$ in $B_{\frac{h}{2}}(x_0)$.
 The contradiction has followed by assuming that there exists  $0<\rho < \tilde{r}$ such that \eqref{cont ndg} holds true, with $\ep$ satisfying $\ep<\frac{M}{4}$
 and $4\ep\leq\frac{CM}{2^\gamma}$. Therefore, if we choose  for example
  \begin{equation*}
  c=\frac{1}{2}\min\left\{\frac{M}{4},\frac{CM}{2^{\gamma+2}}\right\},\end{equation*} inequalities \eqref{ndg} hold true.
\end{proof} 

\begin{lem}
\label{lemma: lower bound J}
Let $z \in \Ga $  be a regular point and let $u$ be a Lipschitz solution of problem \eqref{limitproblemk=2}. Then there exists a constant $C=C(z)>0$ such that for every $ 0<r< \tilde{r}$,
$$J_r(u_i,z) \geq C,\quad i=1,2,$$  where $J_r(u_i,z) $ is defined by \eqref{J(ui)}.
\end{lem}


\begin{proof}
Without loss of generality, we prove the lemma for $i=1$. By Lemma \ref{lemma: both bounds} 
 there exists $c=c(z)>0$ such that for any radius $ r < \tilde{r}<d(z,\partial\Om),$
 there exists $y \in  \overline{ B}_\frac{r}{4}(z)$ such that 
 \begin{equation}
\label{for harnack}
u_1 (y)\geq c\frac{r}{4}.
\end{equation}
 Let $x_0 \in \Ga$ be the closest point in the free boundary to $y,$ $h= \abs{y-x_0}$ and consider $B_{h}(y)$. Note that $h \leq \abs{y-z}\leq r/4$, 
 \begin{equation}\label{balltan}B_{h}(y)  \subset  \Om(u_1)\cap B_r(z),\end{equation} and $B_{h}(x_0)  \subset   B_r(z).$  
 Moreover, since $u_1$ is  Lipschitz continuous in $\Om$, 
there exists $L$ such that $|u(x)-u(y)|\leq L|x-y|$ for any $x,y\in B_{r}(z)$. In particular, we have that  
$$c\frac{r}{4}\leq u_1(y)-u_1(x_0)=u_1(y)\leq L|y-x_0|= Lh,$$ which implies
\begin{equation}\label{hgreaterrrho}\ds\frac{h}{r}\geq \frac{c}{4 L}.
\end{equation}
Next, since \eqref{balltan} holds,  Lemma \ref{onesidepointsregular} implies that $x_0$   is also a regular point and  for  any $x\in B_h(y)\setminus B_\frac{h}{2}(y)$,  
 \begin{equation}\label{znondegencond}
 u_1(x) \geq M d(x,\partial B_h(y))=M(h-|x-y|),
 \end{equation}
 where $M=\ds \frac{u_1(y)C \gamma}{h\, (2^\gamma-1)},$ (see proof for Lemma \ref{onesidepointsregular}).
In particular,  
 for any $s<\frac{h}{2}$,
 \begin{equation*}\sup_{B_s(x_0)}u_1\geq Ms.
  \end{equation*} 
We now note that, from \eqref{for harnack} and  the inequality  $h\leq r/4$, 
\begin{equation}\label{tildeM}
M \geq \frac{c \,  \gamma \,C}{2^{\gamma}-1}=:\tilde{M}, 
\end{equation}
where $\tilde{M}$ depends on $z$ but it is independent of $x_0$, $h$ and $r$.
 Since now we have, for any $s<h/2$
$$
\sup_{B_{s}(x_0)} U \geq \tilde{M} \, s,
$$
where $U=\max\{u_1,u_2\}$, by Lemma \ref{lemma: both bounds} there exists $\tilde{c}=\tilde{c}(\tilde{M})$ such  that,  for any $s<h/2$, 
\begin{equation}
\label{u2nodegenelemJ}
\sup_{B_s(x_0)}u_2\geq \tilde{c}s.
\end{equation}  

We are now in conditions to  apply a Poincar\'{e}-Sobolev type inequality  to $u_1$ (see e.g. \cite[Chapter 4, Lemma 2.8]{han_elliptic_2011} and \cite[Theorem 3]{jiang_zero_2004}). Indeed, we claim that the zero set of $u_1$  has positive density.\\
{\em{Claim:}  There exists $\ep>0$ independent of $h$ such that  
\begin{equation}\label{u2nodegenmeasure}|\{u_1=0\}\cap B_{\frac{h}{2}}(x_0)|\geq |\{u_2>0\}\cap B_{\frac{h}{2}}(x_0)|\geq \ep  |B_{\frac{h}{2}}(x_0)|.\end{equation}
    }\\
{\em Proof of the claim:} Suppose by contradiction that for any $\ep>0$ one has $\left|\{u_2>0\}\cap B_{\frac{h}{2}}(x_0)\right|< \ep \left |B_{\frac{h}{2}}(x_0)\right|$.
Since  $u_2$ is Lipschitz continuous  in $\Om$, 
there exists $L>0$ such that for all $x\in B_{\frac{h}{2}}(x_0)$,
\begin{equation}\label{boundfromlip}u_2(x)\leq L|x-x_0|\leq L\ds\frac{h}{2}.\end{equation}

 Since  $u_2$ is subharmonic, see \eqref{dos}, the mean-value Theorem implies that for any $x\in B_{\frac{h}{4}}(x_0)$,
$$u_2(x)\leq \fint_{ B_{\frac{h}{4}}(x)}u_2(t)dt\leq\frac{1}{\left| B_{\frac{h}{4}}(x)\right|}\int_{\{u_2>0\}\cap B_{\frac{h}{2}}(x_0)}u_2(t)dt
\leq  \frac{\ep \left|B_{\frac{h}{2}}(x_0)\right|}{\left|B_{\frac{h}{4}}(x)\right|}L\frac{h}{2}=\ep 2^{n-1}Lh,$$
which is in contradiction with \eqref{u2nodegenelemJ} with $s=\frac{h}{4}$ for $\ep< \frac{\tilde{c}}{L2^{n+1}}$. This proves \eqref{u2nodegenmeasure}.

Next, to conclude the proof of the lemma, since 
\begin{equation}
\label{uno}
 \frac{1}{r^{2}} \int_{B_{r}(z)} \frac{\vert  \nabla u_1(x) \vert^{2}}{\vert x-z \vert^{n-2}} dx \geq
 \frac{1}{r^{n}} \int_{B_{r}(z)} \vert  \nabla u_1(x) \vert^{2} dx
 \geq \frac{1}{r^{n}} \int_{B_{\frac{h}{2}}(x_0)} \vert  \nabla u_1(x)\vert^{2}dx
\end{equation}
we just need to bound from below the last integral. 

%
Since  \eqref{u2nodegenmeasure} holds true, we can apply  the Poincar\'{e}-Sobolev type inequality 
 to obtain
\begin{equation}
\label{poincare}
 \frac{1}{r^{n}} \int_{B_{\frac{h}{2}}(x_0)} \vert  \nabla u_1(x)\vert^{2}dx\geq  \frac{1}{r^{n}}\frac{1}{C(n,\ep)h^2} \int_{B_{\frac{h}{2}}(x_0)}u_1^2(x)dx.\end{equation}
Finally, by using \eqref{znondegencond} and \eqref{tildeM}, we get

\begin{equation}\label{lastestimhJ}\begin{split}\frac{1}{r^{n}}\frac{1}{C(n,\ep)h^2} \int_{B_{\frac{h}{2}}(x_0)}u_1^2(x)dx
&\geq \frac{1}{r^{n}}\frac{1}{C(n,\ep)h^2} \int_{B_{\frac{h}{2}}(x_0)\cap B_h(y)}\tilde{M}^2(h-|x-y|)^2dx\\&
\geq  \frac{1}{r^{n}}\frac{1}{C(n,\ep)h^2} \int_{B_{\frac{h}{2}}(x_0)\cap B_{\frac{7}{8}h}(y)}\tilde{M}^2(h-|x-y|)^2dx\\&
\geq \frac{  \tilde{M}^2 h^2h^n}{r^{n}C(n,\ep)h^2},
\end{split}\end{equation}
where in the last inequality we have used that $ |B_{\frac{h}{2}}(x_0)\cap B_{\frac{7}{8}h}(y)|\geq \bar{c}\,h^n$ and $ \tilde{M}^2(h-|x-y|)^2\geq \frac{\tilde{M}^2}{64} h^2$ for any $x\in B_{\frac{7}{8}h}(y).$

Putting all together, from \eqref{hgreaterrrho}, \eqref{uno}, \eqref{poincare} and \eqref{lastestimhJ}   we infer that there exists $C=C(z)>0$ such that 
 $$\frac{1}{r^{2}} \int_{B_{r}(z)} \frac{\vert  \nabla u_1(x) \vert^{2}}{\vert x-z \vert^{n-2}} dx \geq C,$$ and this concludes the proof of the lemma.
\end{proof}
The following is an immediate corollary of Lemma \ref{lemma: lower bound J}.

\begin{cor}\label{corlemJpos}
Let $u$ be a Lipschitz solution of problem \eqref{limitproblemk=2}  and let $z \in \Ga$ be a regular point.  Then there exists  a  constant $C=C(z)>0$ such that, for any $ 0<r < \tilde{r}$, 
\begin{equation}
J_r(u,z)\geq C,
\end{equation}
where $J_r(u,z)$ is defined by \eqref{J(u)}.
\end{cor}

\section{Proof of  Theorem \ref{limitpbthm}}\label{Freeboundarycondisec} 
We start with the analysis of the blow up of the solution at regular points.  As in Section \ref{regulapointssec},  throughout this section we will use the notation introduced in \eqref{pos neg} for $u$ solution of \eqref{limitproblemk=2}.
\begin{lem}\label{planeregularpointslem} Let $u$ be a Lipschitz solution of problem \eqref{limitproblemk=2}. Let 
 $0\in\Gamma$ be a regular point. 
   Let $u_r$ denote the blow-up sequence
 $$
u_r(x):=\frac{1}{r}u(rx), \quad  x\in B_2(0),
$$
with $r<d(0,\partial\Om)/2$. Then, $u_r$ admits a uniformly converging subsequence in $B_1(0)$ and   for any converging subsequence   $u_{r_j}(x) =u(r_jx)/r_j$, $j\in\N$,  there exist $\alpha,\,\beta>0$ and a unit vector $n$, such that,
 \begin{equation}\label{alphabetaj0}J_0(u)=c_n\alpha^2\beta^2,\end{equation}
 where $J_0(u)$ is defined as in \eqref{zero0}, and 
 as $j\to+\infty$, 
\begin{equation}\label{ubarpalbe}u_{r_j}(x)\to\alpha  <x, n>^+ - \,\beta<x,n>^-,\end{equation}
uniformly  in $B_1(0)$.
\end{lem}
\begin{proof}
Since $u$ is Lipschitz continuous in $\Om$, the sequence $\{u_r\}$ is uniformly bounded in $C^{0,1}(B_2(0)).$ Therefore, by Ascoli-Arzela,  there exists
a subsequence $\{u_{r_j}\}$  and a Lipschitz function $\bar u$, such that, as $j\to+\infty$,    $u_{r_j}\to \bar u$ uniformly in $B_1(0)$ and weakly in $H^1(B_1(0))$.
In particular, 
\begin{equation}\label{limiggrad}\int_{B_1(0)}|\nabla \bar u|^2\,dx\leq \liminf_{j\to+\infty}\int_{B_1(0)}|\nabla u_{r_j}|^2\,dx.\end{equation}
 We will prove that for any $s\in(0,1)$, 
\begin{equation}\label{monotformulalimse7}J_s(\bar{u})=J_0(u)>0,\end{equation}
where $J_r(u)$ is defined as in  \eqref{J(u)}.
 For that, we truncate $(u_1)_{r_j}=u_1(r_jx)/r_j$ at level $\ep$ and $h$, for  $0<\ep<h$, by considering $w_{\ep,h}:=\min\{\max\{(u_1)_{r_j},\ep\},h \}$. 
Since each $(u_1)_{r_j}$ is subharmonic (i.e., $\Delta (u_1)_{r_j} =\mu_{r_j} \geq 0$, in the sense of distributions, and $\mu_j$ is a Radon measure)  and  Lipschitz continuous, then we have $w_{\ep,h}\nabla (u_1)_{r_j} \in L^{\infty}(B_1(0))$ and the product rule $\div \, (w_{\ep,h}\nabla (u_1)_{r_j})=w_{\ep,h} \mu_{r_j} + \nabla w_{\ep,h} \cdot \nabla (u_1)_{r_j}$ holds in the sense of distributions. Moreover, we can integrate by  parts (see \cite{ctz,cct}) in $B_1(0)$:
\begin{equation}
\label{util}
\int_{B_1(0)} \div (\, w_{\ep,h}\nabla (u_1)_{r_j}) = \int_{\partial B_1(0)} (w_{\ep,h}\nabla (u_1)_{r_j} \cdot \nu)_{tr} d\mathcal{H}^{n-1},
\end{equation}
where $(w_{\ep,h} \nabla (u_1)_{r_j} \cdot \nu)_{tr} \in L^{\infty} (\partial B_1(0))$ is the normal trace of the vector field $w_{\ep,h}\nabla (u_1)_{r_j}$ and which satisfies
\begin{equation} 
 (w_{\ep,h}\nabla (u_1)_{r_j} \cdot \nu)_{tr} \leq \norm{ w_{\ep,h}\nabla (u_1)_{r_j}} {L^{\infty}(B_1(0))} \leq h \norm{\nabla (u_1)_{r_j}}{L^{\infty}(B_1(0))} \leq hL.
\end{equation} 
From \eqref{util}, and since $\nabla w_{\ep,h} =0$ a.e. in  $B_1(0) \cap \{(u_1)_{r_j} \geq h\}$ and in $B_1(0) \cap \{(u_1)_{r_j} \leq \ep \}$, we obtain

\begin{equation}\begin{split}
0&\le \int_{B_1(0)} w_{\ep,h}d\mu_{r_j}=-\int_{\ep \leq (u_1)_{r_j}\leq h\}\cap B_1(0)}|\nabla (u_1)_{r_j}|^2\,dx+\int_{\partial B_1(0)} (w_{\ep,h} \nabla (u_1)_{r_j} \cdot \nu)_{tr}d\mathcal{H}^{n-1}
\\&\le -\int_{\{\ep \leq(u_1)_{r_j}\leq h\}\cap B_1(0)}|\nabla (u_1)_{r_j}|^2\,dx+ Ch,\end{split}\end{equation}
with $C= L\mathcal{H}^{n-1} (\partial B_1(0))$. Using the Lebesgue Dominated Convergence Theorem we let $\ep$ go to 0, obtaining:
\begin{equation}\label{gradientstrip}
\int_{\{0\leq (u_1)_{r_j}\leq h\}\cap B_1(0)}|\nabla (u_1)_{r_j} |^2\,dx\leq Ch.
\end{equation}
Similarly, one gets
\begin{equation}\label{gradientstripbis}
\int_{\{0\leq (u_2)_{r_j}\leq h\}\cap B_1(0)}|\nabla (u_2)_{r_j}|^2\,dx\leq Ch.
\end{equation}
 From \eqref{gradientstrip} and \eqref{gradientstripbis}, we obtain
\begin{equation}
\label{navidad1}
\int_{\{|u_{r_j}|\leq h\}\cap B_1(0)}|\nabla u _{r_j} |^2\,dx\leq Ch.
\end{equation}
Next, for $j$ large enough, the set  $\{|u_{r_j}|>0\}$ contains  $\{|\bar u|>h\}$.  Moreover, since $u_{r_j}$ is a Lipschitz viscosity solution of 
\begin{equation}\label{blowupsolutions}\begin{cases}
\Fn (u_{r_j})= 0 &\text{in }\{u_{r_j}>0\}\cap B_2(0)\\
 \Fp (u_{r_j})= 0 &\text{in }\{u_{r_j}<0\}\cap B_2(0),
\end{cases}
\end{equation}
by the interior $C^{2,\alpha}$ estimates for the operators $\mathcal{F}^\pm$ (see  Remark \ref{FpFnremark}), up to a subsequence, 
$\nabla u_{r_j}\to \nabla \bar u$ uniformly in $\{|\bar u|>h\}\cap B_1(0)$ as $j\to +\infty$, and thus,
\begin{equation}
\label{navidad2}
\lim_{j\to+\infty}\int_{\{|u_{r_j}|>h\}\cap B_1(0)}|\nabla u_{r_j}|^2\,dx=\int_{\{|\bar u|>h\}\cap B_1(0)}|\nabla \bar u|^2\,dx.
\end{equation}
By \eqref{navidad1} and \eqref{navidad2} we infer that, for any $h>0$,
$$\limsup_{j\to+\infty}\int_{B_1(0)}|\nabla u_{r_j}|^2\,dx\le \int_{B_1(0)}|\nabla \bar u|^2\,dx+Ch,$$
which combined with \eqref{limiggrad}  yields, letting $h\to0$,
\begin{equation}\label{finally}
\lim_{j\to+\infty}\int_{B_1(0)}|\nabla u_{r_j}|^2\,dx=\int_{B_1(0)}|\nabla \bar u|^2\,dx.
\end{equation}
By \eqref{finally}, $|\nabla u_{r_j}|^2\to |\nabla \bar u|^2$ a.e. in $B_1(0)$. Since in addition $|\nabla u_{r_j}|^2/|x|^{n-2}\le L^2 /|x|^{n-2}\in L^{1}(B_1(0))$, by the Dominated Convergence Theorem we infer that, for any $s\in(0,1)$,
\begin{equation}
\label{2}
 \lim_{j\to+\infty}J_s(u_{r_j})=J_s(\bar{u}).
\end{equation}
Next, by Corollary \ref{corlemJpos} and  Remark \ref{remark: lower bound}
\begin{equation}
\label{zero}
\lim_{r\rightarrow 0^+} J_r(u)=J_0(u)>0.
\end{equation}
Let $s\in(0,1)$. A change of variables yields:
\begin{equation}
\label{1}
J_s(u_{r_j})=J_{sr_j}(u).
\end{equation}
Therefore, by \eqref{2}-\eqref{1},
 for any $s\in(0,1)$,
\begin{equation*}
\label{zerozero}
J_s(\bar u)=\lim_{j\to+\infty} J_s(u_{r_j})=\lim_{j\to+\infty} J_{sr_j}(u)=J_0(u)>0,
\end{equation*}
 which gives \eqref{monotformulalimse7}.
The conclusion of the lemma follows from Theorem \ref{blowupthm}.
\end{proof}

\begin{cor}\label{flatnessprop}
Under the assumptions of Lemma \ref{planeregularpointslem}, 
 for any $\ep>0$ there exists $J\in\N$ such that for any $j\ge J$,   all the level sets of $u_{r_j}$ in $B_1(0)$ are $\ep$-flat, in the sense that
for any $\lambda\in\R$, we have 
\begin{equation}\label{flatnesspropeq}\{u_{r_j}=\lambda \}\cap B_1(0)\subset\{x\in\R^n\,|\,d(x,\Pi_\lambda)<c\ep \},\end{equation}
for some $c>0$ independent of $ \lambda,\,r_{j}$ and $\ep$, 
where 
$$
\Pi_\lambda=
\begin{cases}
\{ <x, n>=\lambda /\alpha\}&\quad\text{if }\lambda >\ep,\\
\{ <x, n>=0\}&\quad\text{if }\lambda \in[-\ep,\ep], \\
\{ <x, n>=\lambda /\beta\}&\quad\text{if }\lambda <-\ep.
\end{cases}
$$
and $n, \alpha \: \mbox{and} \:\beta$ are as in \eqref{ubarpalbe}.
\end{cor}
\begin{proof}
By Lemma \ref{planeregularpointslem},
for any $\ep>0$, there exists $J\in\N$ such that for any $j\ge J$, 
\begin{equation}\label{uclosetoplanepropflat}\left|u_{r_j}(x)- \alpha  <x, n>^+ + \,\beta<x,n>^-\right|<\frac{\ep}{2},\end{equation} for all $x\in B_1(0)$.
Let  $\lambda>\ep $, then by \eqref{uclosetoplanepropflat}, $\text{if }<x, n>\ge\frac{\lambda+\ep }{\alpha}$ then
$$u_{r_j}(x)\geq \lambda +\ep -\frac{\ep}{2}= \lambda +\frac{\ep}{2}>\lambda,$$
and  $\text{if }<x, n>\le\frac{\lambda-\ep}{\alpha}$, then
$$u_{r_j}(x)\leq \lambda -\frac{\ep}{2}<\lambda.$$
We infer that 
\begin{equation}\label{flatnesspropeq1}\{u_{r_j}=\lambda \}\cap B_1(0)\subset \left\{x\in\R^n\,|\, \frac{-\ep }{\alpha}\le\: <x, n>-\frac{\lambda }{\alpha}\leq \frac{\ep }{\alpha}\right\}.\end{equation}
Similarly, if $\lambda \in[-\ep ,\ep ]$, 
\begin{equation}\label{flatnesspropeq2}\{u_{r_j}=\lambda \}\cap B_1(0)\subset \left\{x\in\R^n\,|\, \frac{-2\ep }{\beta}\le \: <x, n>\: \leq \frac{2\ep }{\alpha}\right\},\end{equation}
 and if $\lambda<-\ep$, 
\begin{equation}\label{flatnesspropeq3}\{u_{r_j}=\lambda \}\cap B_1(0)\subset \left\{x\in\R^n\,|\, \frac{-\ep }{\beta}\le \:<x, n>-\frac{\lambda }{\beta}\leq \frac{\ep }{\beta}\right\}.\end{equation}
Inclusions \eqref{flatnesspropeq1},  \eqref{flatnesspropeq2} and  \eqref{flatnesspropeq3} give 
\eqref{flatnesspropeq} with 
$$c=2\max\{\alpha^{-1}, \beta^{-1}\}.$$
\end{proof}
By Lemma \ref{planeregularpointslem}  we know that if $0\in\Gamma$ is a regular point, then the blow up sequence $u(rx)/r$ admits a subsequence converging to a two-plane solution of the form \eqref{ubarpalbe}.
We next show that if there is a tangent ball from one side to $\Gamma$ at 0, then the full sequence $u(rx)/r$ converges to a two-plane solution and therefore $u$ has an asymptotic linear behavior  at 0. 

\begin{lem}\label{planeballconcor} Let $u$ be a Lipschitz solution of problem \eqref{limitproblemk=2}.
Let $0\in\Gamma$. Assume that there exists a tangent ball $B$ from one side to $\Gamma$ at 0. Then, 
 there exist   $\alpha,\,\beta>0$ such that 
\begin{equation}\label{planeballconcorformula}u(x)=\alpha  <x, {\bf \nu}>^+ - \,\beta<x,{\bf \nu}>^-+o(|x|),\end{equation}
where $\nu$ is the normal vector of $B$ at 0 pointing inward $ \{u>0\}$.
\end{lem}

\begin{proof}
By Lemma \ref{onesidepointsregular}, 0 is a regular point.
Consider the sequence  $u_r(x)=\frac{1}{r} u(rx)$, for $r$ small enough.  By  Lemma \ref{planeregularpointslem}, there exist  a subsequence $\{u_{r_j}\}$, a unit vector $n$ and $\alpha, \beta$ positive constants  satisfying \eqref{alphabetaj0}, such that as $j\to+\infty$,
$$
u_{r_j}(x)\to \overline u(x):= \alpha <x, n>^+ - \,\beta<x,n>^-, 
$$
uniformly in $ B_1(0)$. 

Assume, without loss of generality, that there exists a ball $B_{r_0}(y) \subset \Om(u_1)$ such that $0\in\partial B_{r_0}(y)$. Let $\nu$ be the normal vector of $B_{r_0}(y)$ at 0 pointing inward $  \Om(u_1)$. \\
\noindent{\em Claim 1 :}
We claim that $\nu=n$.  \\
\noindent{\em Proof of Claim 1 :}
Indeed, suppose by contradiction that $\nu\neq n$. Then, there exists $x_0\in B_{1}(0)$ such that   for any $j,$ $<r_j\,x_0, \nu>\: >0$ and $<r_j\,x_0, n>\: <0$. 
Fix $J\in\N$ such that the sequence of points $\{r_jx_0\}$ satisfies $r_jx_0\in B_{r_0}(y)\subset \Om(u_1)$ for all $j\ge J$, then
$$\frac{u(r_jx_0)}{r_j}>0\quad\text{for all }j\ge J.$$ 
Passing to the limit as $j\to+\infty$, we get $$\overline u(x_0)\geq 0.$$
On the other hand, 
$$\overline u(x_0)=- \,\beta<x_0,n>^-<0.$$ This is a contradiction. Therefore we must have $\nu=n$.

 We now proceed to show that the full sequence $u_r$ converges to $\overline u$. Let  $\overline u$ and  $\overline v$ be the limits of two converging subsequences of the sequence $\{u_r\}$, then we must have 
$$\overline u= \alpha_1 <x, \nu>^+ - \,\beta_1<x,\nu>^-$$ and 
$$\overline v= \alpha_2<x, \nu>^+ - \,\beta_2<x,\nu>^-.$$

\noindent{\em Claim 2 :}
We claim that in addition that
\begin{equation}\label{alpha1=aplha2claim}\alpha_1=\alpha_2\quad\text{and}\quad\beta_1=\beta_2.\end{equation}
\noindent{\em Proof of Claim 2 :}  To prove this claim,  we will construct a barrier to bound $u_1$ from below. 
Let $\phi$ be the solution of 
\begin{equation*}
\begin{cases}
\Fn(\phi)=0&\text{in }B_{r_0}(y)\setminus B_\frac{r_0}{2}(y)\\
\phi=1&\text{on }\partial B_\frac{r_0}{2}(y)\\
\phi=0&\text{on }\partial B_{r_0}(y).
\end{cases}
\end{equation*}
By the comparison principle, for any $x\in  B_{r_0}(y)\setminus B_\frac{r_0}{2}(y)$
\begin{equation}\label{uleqphi1lem73}u_1(x)\geq c_0\phi(x),\quad\text{with }c_0=\min_{\partial B_\frac{r_0}{2}(y)}u_1.
\end{equation}
For $k\ge 0$ such that $2^{-k}<r_0/2$, let
$$\tilde{m}_k:=\sup\{m\,|\,u_1(x)\ge m \phi(x)\text{ in }B_{2^{-k}}(0)\cap B_{r_0}(y)\}.$$
Notice that the sequence $\tilde{m}_k$ is increasing. Moreover, by \eqref{uleqphi1lem73}, $\tilde m_k\ge c_0$ for any $k$.  Let 
$$\tilde{m}_\infty:=\sup_k \tilde{m}_k=\lim_{k\to+\infty}\tilde{m}_k.$$ 
Since $u_1$ is Lipschitz continuous, $\tilde{m}_\infty<+\infty$.

By Lemma \ref{fondamentalsollemF}, there exists $\sigma>0$ independent of $r_0$ such that, for any $x\in  B_{r_0}(y)\setminus B_\frac{r_0}{2}(y)$,
\begin{equation}\label{phibehavlemsamangle}\phi(x)=\frac{\sigma}{r_0}<x,\nu>+o(|x|).\end{equation}
Set 
$$m_k:=\frac{\sigma}{r_0}\tilde{m}_k\quad\text{and}\quad m_\infty:=\frac{\sigma}{r_0}\tilde{m}_\infty,$$ 
By the definition of $m_k$ and \eqref{phibehavlemsamangle},  
for $x\in B_{2^{-k}}(0)\cap B_{r_0}(y)$ we have
$$u_1(x)\geq \tilde m_k \phi(x) = m_k  <x, \nu>+o(|x|).$$ 
This implies that $\alpha_1\geq m_\infty$. 
Assume by contradiction that $\alpha_1> m_\infty$. 
 We will show that in this case, there exists $\ep > 0$ and a  sequence  $k_j$ such that, for $j$ large enough,
\begin{equation}
\label{contra}
u_1 - \left (\tilde{m}_{k_j}+ \ep \right) \phi(x) \geq 0, \:\text{for all } x \in B_{2^{-k_j}}(0) 
\cap B_{r_0}(y),
\end{equation}
which is in contradiction with the definition of $\tilde{m}_{k_j}$.
For that, if $\alpha_1> m_\infty$,   there exists a sequence $r_j\to 0$ as $j\to+\infty$ such that if $y_j=r_j\nu$, then, for some $\delta_0>0$,
$$u_1(y_j)- \tilde{m}_\infty  \phi(y_j)\geq \delta_0 r_j.$$
Let $\tilde m_{k_j}$ be a subsequence converging to $\tilde m_\infty$ as $j\to+\infty$ such that,  up to eventually consider a subsequence of $\{r_j\}$, one has that $2r_j\leq 2^{-k_j}$.  Then $B_{r_j}(y_j)\subset B_{2^{-k_j}}(0)\cap B_{r_0}(y)$ and
since $\tilde m_{k_j}\le  \tilde{m}_\infty$, 
$$u_1(y_j)- \tilde{m}_{k_j} \phi(y_j)\geq \delta_0 r_j.$$
 By the definition of $\tilde{m}_{k_j} $  the function $u_1-\tilde{m}_{k_j} \phi$ is positive in $B_{2^{-k_j}}(0)\cap B_{r_0}(y)$ and  by Proposition \ref{Fnproperties} it  satisfies 
 
$$\ppuccin(u_1-\tilde{m}_{k_j}\phi)\le  \Fn(u_1-\tilde{m}_{k_j}\phi)\le\Fn(u_1)-\Fn(\tilde{m}_{k_j}\phi)=0,$$
and 
$$\ppuccip(u_1-\tilde{m}_{k_j}\phi)\ge \Fn(u_1)-\Fn(\tilde{m}_{k_j}\phi)=0.$$
Therefore, since $B_{r_j}(y_j)\subset B_{2^{-k_j}}(0)\cap B_{r_0}(y)$, by the Harnack inequality,
\begin{equation}\label{harnacklem73}u_1(x)-\tilde{m}_{k_j} \phi(x)\geq c\delta_0 r_j\quad\text{for  }x\in B_\frac{r_j}{2}(y_j),\end{equation} where $c>0$ is a universal constant. 
By a barrier argument we see that there exist $\delta_j$ and $\tilde{c}>0$ (independent of $j$) such that: 
\begin{equation}\label{deltaeqlem73} u_1(x)-\tilde{m}_{k_j} \phi(x)\geq \tilde{c}d(x,\partial B_{r_0}(y))\quad \text{for }x\in B_{\delta_j}(0)\cap B_{r_0}(y).
\end{equation}
 Indeed, let $z\in B_\frac{r_j}{4}(y_j)\cap \partial B_{r_0-r_j}(y)$ and let $w$ be the closest point to $z$ in  $\partial B_{r_0}(y)$, that is $w\in \partial B_{r_0}(y)\cap \partial  B_{r_j}(z)$.
By \eqref{harnacklem73}, $u_1(x)-\tilde{m}_{k_j} \phi(x)\geq c\delta_0 r_j$ for $x\in B_\frac{r_j}{4}(z)$.
Let $\psi(x)$ be the solution to 
\begin{equation*}
\begin{cases}
\ppuccin(\psi)=0&\text{in }B_{r_j}(z)\setminus B_\frac{r_j}{4}(z)\\
\psi=c\delta_0r_j &\text{on }\partial B_\frac{r_j}{4}(z)\\
\psi=0&\text{on }\partial B_{r_j}(z).
\end{cases}
\end{equation*}
By Lemma \ref{fondamentalsollem}, $\psi(x)=\frac{c\delta_0r_j}{4^\gamma-1}\left(\frac{r_j^\gamma}{|x-z|^\gamma}-1\right)$, $\gamma=\Lambda(n-1)/\lambda-1$  and 
$\psi(x)\ge \frac{c\delta_0\gamma}{4^\gamma-1}(r_j-|x-z|)$. In particular, for all points $x$ in $B_{r_j}(z)\setminus B_\frac{r_j}{4}(z)$ belonging to the segment
from $z$ to $w$ we have that $\psi(x)\ge \tilde{c}(r_j-|x-z|)=\tilde{c} d(x,\partial B_{r_0}(y))$, with $\tilde{c}:=\frac{c\delta_0\gamma}{4^\gamma-1}$. Letting $z$ vary in $B_\frac{r_j}{4}(y_j)\cap \partial B_{r_0-r_j}(y)$, we get \eqref{deltaeqlem73}.

For every $x \in B_{\delta_{j}}(0) \cap B_{r_0}(y)$, we have
\begin{eqnarray*}
u_1 - \tilde{m}_{k_j} \phi(x) &\geq& \tilde{c} d (x, \partial B_{r_0}(y))\\
  &\geq& 2\ep \phi(x)
\end{eqnarray*}
 for some $\ep>0$,
 hence:
\begin{equation}
\label{yaporfavor}
u_1 - (\tilde{m}_{k_j}+ 2\ep) \phi(x) \geq 0, \text{for all } x \in B_{\delta_{j}}(0) \
\cap B_{r_0}(y), \text{ for every } j
\end{equation}
Let $j_0$ be such that $0<\tilde{m}_{\infty} - \tilde{m}_{k_j} < \ep$ for all 
$j \geq j_0$. Given $j_0$, there exists an integer $j_1 \geq j_0$ such that if $j \geq j_1$ then $B_{2^{-k_j}}(0) \cap B_{r_0} (y) \subset B_{\delta_{j_0}} \cap B_{r_0} (y)$. From \eqref{yaporfavor},
\begin{equation*}
u_1 - (\tilde{m}_{k_{j_0}}+ 2\ep) \phi(x) \geq 0, \text{ for all } x \in B_{2^{-k_j}}(0) \cap B_{r_0}(y), \quad j \geq j_1.
\end{equation*}
Thus, for  $j \geq j_1$ we have
\begin{eqnarray*}
u_1 - \left(\tilde{m}_{k_j}+ \ep\right) \phi(x) 
 &=&  u_1 - \left(\tilde{m}_{k_{j_0}} + \ep\right)\phi(x) - (\tilde{m}_{k_j}- \tilde{m}_{k_{j_0}})\phi(x)\\
 &\geq & u_1 -\left (\tilde{m}_{k_{j_0}} +  \ep\right)\phi(x) - \ep\phi(x)\\
 &\geq &  u_1 - (\tilde{m}_{k_{j_0}} + 2\ep )\phi(x) \\&\geq& 0, \text{ for all } x \in B_{2^{-k_j}}(0)\cap B_{r_0}(y),
\end{eqnarray*} 
which contradicts the definition of $\tilde{m}_{k_j}$.

By  \eqref{alpha1=aplha2claim}  we infer that  $\overline u=\overline v$. Since any  subsequence of  $\{u_r\}$ converges to the same function, we deduce that the whole sequence  $\{u_r\}$ converge, 
 as $r\to 0$, uniformly in $B_1(0)$
to the  limit function
$$\overline u= \alpha <x, \nu>^+ - \,\beta<x,\nu>^-$$for some $\alpha,\beta>0$.
This means that for any $\ep>0$, there exists $\rho>0$ such that for any $0<r<\rho$ and any $x\in \overline{B_1(0)}$, then
\begin{equation*}
\label{garganta}
\left| u_r(x)- \alpha <x, \nu>^+ + \,\beta<x,\nu>^-\right|<\ep.
\end{equation*}
Now,  fix $\ep>0$ and let $\rho>0$ be defined as above. Fix $z \in B_{\rho}(0)$ and let $r = |z|.$ Since 
$\frac{z}{r} \in \overline{B_1(0)}$  we have
$$\left| u_r\left(\ds\frac{z}{r}\right)- \alpha <\ds\frac{ z}{r}, \nu>^+ + \,\beta<\frac{ z}{r},\nu>^-\right|<\ep,$$
that is,
\begin{equation}
\label{ya}
\left|u(z)- \alpha <z, \nu>^+ + \,\beta<z,\nu>^-\right|<\ep r=\ep|z|,
\end{equation}
which proves \eqref{planeballconcorformula}.
\end{proof}
Notice that in Lemma \ref{planeballconcor} we did not use that $u$ is a viscosity solution of $\Fn(u)\le 0\le \Fp(u)$ in  $\Om$. In the next theorem we show that if $u$ satisfies in the viscosity sense these two differential inequalities and has the asymptotic linear behavior   \eqref{planeballconcorformula}, then we must have $\alpha=\beta$. 
\begin{thm}\label{alpha=betathm}
Let $u$ be a Lipschitz solution of problem \eqref{limitproblemk=2}. Let $0\in\Gamma$. 
 Assume that  exist
 $\alpha,\,\beta>0$ and a unit vector ${\bf \nu}$ such that 
\begin{equation}\label{ulikepalnethmsameang}u(x)=\alpha  <x, {\bf \nu}>^+-\beta <x, {\bf \nu}>^-+o(|x|).\end{equation}
Then $\alpha=\beta.$
\end{thm}
\begin{proof}

We first  prove that $\beta\geq \alpha$. 
We argue by contradiction, assuming that  $\beta<\alpha$.
Fix $h>0$ and let   $y:=h {\bf \nu}.$ Notice that $|y|=h$ and $\nu$ is the interior normal unit vector  to $B_{h}(y)$ at 0. 
Consider the function 
$$\phi(x)=c\left(\frac{h^\gamma}{|x-y|^\gamma}-1\right),\quad x\neq y,$$ with $\gamma=\frac{\Lambda(n-1)-\lambda}{\lambda}$ and $c>0$.
Then, by Lemma \ref{fondamentalsollem}, 
$$\begin{cases}
\phi(x)>0&\text{if }|x-y|<h\\
\phi(x)<0&\text{if }|x-y|>h\\
\frac{\partial}{\partial \nu}\phi(x)=\frac{c\gamma}{h}&\text{if }|x-y|=h\\
\ppuccin(\phi)(x)=0 &\text{if }x\neq y.
\end{cases}
$$
Since  $\beta<\alpha$, there exists  $\ep>0$ such that $\beta+\ep<\alpha-\ep$. Then, we choose $c>0$ such that 
\begin{equation}\label{chosecalphabeta}\beta+\ep<\frac{\partial}{\partial \nu}\phi_{|_{\partial B_h(y)}}=\frac{c\gamma}{h}<\alpha-\ep.\end{equation}
We want to prove that with this choice of $c$, $\phi\leq u$ in a neighborhood of $0$, for $h$ small enough. 
In order to prove it, we first show that 
$$\phi\leq u\quad\text{on }\partial B_{(1-s)h}(y),$$
for $h$ and $s$ small enough.

Observe that the point $w_1=sh\nu$ belongs to $\partial B_{(1-s)h}(y)$. Moreover,  for any $x\in \partial B_{(1-s)h}(y)$,  
we have that
\begin{equation*}\begin{split}<x,\nu>&\geq <w_1,\nu>= sh,
\end{split}
\end{equation*} 
from which we get
\begin{equation}\label{boudatyusameanfllem}u(x)=\alpha<x,\nu>^+-\beta<x,\nu>^-+o(|x|)=\alpha<x,\nu>^++o(|x|)\geq \alpha sh+ o(h).\end{equation}

Now, let us compute $\phi$ on $ \partial B_{(1-s)h}(y)$. If $|x-y|=(1-s)h$ and $c$ satisfies \eqref{chosecalphabeta}, then
\begin{equation}\label{boudatyusameanfllemphi}\begin{split}\phi(x)=c\left(\frac{1}{(1-s)^\gamma}-1\right)&=\frac{c}{(1-s)^\gamma}\left(1-(1-s)^\gamma\right)
\\&=\frac{c}{(1-s)^\gamma}\left(\gamma s+o(s)\right)\\&\leq \frac{(\alpha-\ep)sh}{(1-s)^\gamma}+o(s)h.\end{split}\end{equation}
Let  $s>0$  be  so small that
$$\frac{(\alpha-\ep)s}{(1-s)^\gamma}+o(s)<\left( \alpha-\frac{\epsilon}{2}\right)s.$$ For such  $s$, let $h$ be so small that 
$$\alpha sh+ o(h)> \left(\alpha-\frac{\epsilon}{2}\right)sh.$$ Then,  
comparing \eqref{boudatyusameanfllem} with \eqref{boudatyusameanfllemphi}, we see that $\phi(x)<u(x)$,  for any $x\in \partial B_{(1-s)h}(y)$.

Next, let us prove that $\phi> u$ on $\partial B_{h(1+s)}(y)$, for suitable small $h$ and $s$. 
Let $w_2=-sh{\bf \nu}$, then  $w_2$ belongs to  $ \partial B_{h(1+s)}(y)$. Moreover, if 
  $x\in \partial B_{h(1+s)}(y)$, then
$$<x,{\bf \nu}>\geq <w_2,{\bf \nu}>=-sh.$$ Therefore, if $x\in \partial B_{h(1+s)}(y)$, and $<x,\nu>\geq 0$, then 
\begin{equation}\label{u2graterphiplane}u(x)=\alpha<x,\nu>^++o(|x|)\geq o(h),
\end{equation} 
and if $<x,\nu>\leq 0$, then 
\begin{equation}\label{u2graterphiplanebis}u(x)=-\beta<x,\nu>^-+o(|x|)\geq -\beta sh +o(h).
\end{equation} 
Let us now compute the value of $\phi$ on $ \partial B_{h(1+s)}(y)$.  
If $|x-y|=h(1+s)$, and $c$ satisfies \eqref{chosecalphabeta}, then  
\begin{equation}\label{boudatyusameanfllemphibeta}\phi(x)=-\frac{c}{(1+s)^\gamma}\left((1+s)^\gamma-1\right)=-\frac{c}{(1+s)^\gamma}\left(\gamma s+o(s)\right)\leq -\frac{(\beta+\ep)sh}{(1+s)^\gamma}+o(s)h. \end{equation}
Let $s$ be so small that 
$$-\frac{(\beta+\ep)s}{(1+s)^\gamma}+o(s)\leq-\left(\beta+\frac{\ep}{2}\right)s.$$
For such  $s$, let $h$ be so small that 
$$-\beta sh +o(h)\geq -\left(\beta+\frac{\epsilon}{2}\right)sh.$$ 
 Then, by \eqref{u2graterphiplane},  \eqref{u2graterphiplanebis} and \eqref{boudatyusameanfllemphibeta}, $\phi<u$ on $ \partial B_{h(1+s)}(y)$. 
Putting all together, we have proven that there exist $s,h>0$ such that
$$\phi<u\quad\text{on }\partial( B_{h(1+s)}(y)\setminus  B_{h(1-s)}(y)).$$
Since, in addition
$$\ppuccin(\phi)=0\geq\puccin(u)\geq \pn(u)\quad\text{on } B_{h(1+s)}(y)\setminus  B_{h(1-s)}(y),$$
the comparison principle combined with the strong maximum principle implies
$$\phi<u\quad\text{in } B_{h(1+s)}(y)\setminus  B_{h(1-s)}(y),$$
which gives a contradiction at $x=0$.

We conclude that we must have $\alpha\leq \beta.$ Arguing similarly as before and using that $\puccip(u)\geq0$ in $\Om$, one can prove that  $\alpha\geq \beta$
and this concludes the proof of the theorem.
\end{proof}

\subsection{Proof of Theorem \ref{limitpbthm}}
Theorem  \ref{limitpbthm} is a corollary of Lemma \ref{planeballconcor} and Theorem \ref{alpha=betathm}.



\section{Proof of Theorem \ref{C1alphaGammathm}}\label{C1alphasec}
Consider  the following  two phase free boundary problem:
\begin{equation}
\label{freebddproblem}
\begin{cases}
\puccin (u) =0\quad\text{in }\Om(u^+)\\
\puccip (u) = 0 \quad\text{in } \Om(u^-)\\
\frac{\partial u^+}{\partial \nu_+}=\frac{\partial u^-}{\partial \nu_-} \quad\text{on } \partial \Om(u^+),
 \end{cases}
 \end{equation}
 where $\nu_\pm$ is the inner normal vector to $\Om(u^\pm)=\{u^\pm>0\}$.  \\
 By Theorem \ref{limitpbthm} we know that any Lipschitz solution to  \eqref{limitproblemk=2} satisfies in the viscosity sense  \eqref{freebddproblem} in $\Om$. 
 Let us recall the definition of viscosity solution of the problem \eqref{freebddproblem} in a given domain $D\subset\R^n$, see  \cite{caffarelli_geometric_2005} for more details. 
 \begin{defn} \label{visco def}Let $u$ be a continuous function in $D$. We say that $u$ is a viscosity solution of the problem \eqref{freebddproblem} in $D$, if the following holds.
 \begin{itemize}
 \item[i)] $u$  satisfies in the viscosity sense 
 \begin{equation*}
\begin{cases}
\puccin (u) =0\quad\text{in }\{u>0\}\cap D\\
\puccip (u) = 0 \quad\text{in }\{u<0\}\cap D.\\
 \end{cases}
 \end{equation*}
 \item[ii)]
 If there exists a tangent ball at $x_0 \in \partial \{u>0\}\cap D$, $B$,  such that either $B \subset \{u>0\}\cap D$ or $B \subset \{u<0\}\cap D$, then
 $$
 u(x)=\alpha <x-x_0,\nu_+>+o(|x-x_0|)
 $$  
 with $\alpha>0$ and  $\nu_+$ the normal vector to  $\partial B$ at $x_0$ pointing inward to $\{u>0\}\cap D$.
 \end{itemize}
 \end{defn}
 
In this section we prove that for any viscosity  solution to the free boundary problem \eqref{freebddproblem}  the following holds: if the free boundary is flat around 0, meaning that it can be trapped in  a small neighborhood of the graph of a Lipschitz function, then in a neighborhood of 0, it is  a $C^{1,\alpha}$ surface.  Theorem \ref{C1alphaGammathm} 
will follow  as  a corollary of this result. 
To prove that flatness implies  $C^{1,\alpha}$, we follow the classical sup-convolution method  developed by Caffarelli in the papers 
\cite{caffarelli_harnack_1987,caffarelli_harnack_1989} for the Laplace operator and extended by Wang \cite{wang_regularity_2000,wang_regularity_2002} to fully-nonlinear elliptic operators. Problem \eqref{freebddproblem} differs from the one studied in  \cite{wang_regularity_2000,wang_regularity_2002} since $u$ satisfies two different equations in 
$\Om(u^+)$ and $\Om(u^-)$. However the regularity theory developed in those papers can be  extended to our problem and some simplifications arise due to the 
specific free boundary condition here considered:   $\frac{\partial u^+}{\partial \nu_+}=\frac{\partial u^-}{\partial \nu_-}$.

Following the classical theory, we first prove that Lipschitz free boundaries are $C^{1,\alpha}$ and then we prove that flat free boundaries are Lipschitz.
 
\subsection{Lipschitz free boundaries are $C^{1,\alpha}$}
For $r>0$, let  $\mathcal{C}_r$  be the cylinder defined as $\mathcal{C}_r:=B'_r(0)\times (-r,r)$, where $B'_r(0)$ is the ball centered at 0 of radius $r$  of $\real^{n-1}$.

\begin{prop} 
\label{Caffa1}
Let $u$ be a viscosity solution of the problem \eqref{freebddproblem} in $\mathcal{C}_1=B'_1(0)\times (-1,1)$. Assume that $0 \in \partial \Om(u^+)$ and that $$\mathcal{C}_1\cap \Om(u^+)=\{(x', x_n)\,|\,x_n>g(x')\}$$ where $g$  is a Lipschitz continuous function. Then in $ B'_{\frac{1}{2}}(0)$, $g$ is a $C^{1,\alpha}$-function, for some $0<\alpha\le 1$.  
\end{prop}

\begin{proof}

 The proof of the   proposition   follows by  
 \cite{caffarelli_harnack_1987} (see also \cite{caffarelli_geometric_2005}) and 
  \cite{wang_regularity_2000}. As already pointed out, even though we have different operators on each side of the free boundary, the classical regularity theory still applies. For completion of this paper, we will sketch the main parts of the method highlighting the parts that are simplified in our problem due to the free boundary condition 
in \eqref{freebddproblem}.

 {\bf Step 1: Existence of a cone of monotonicity.} 
 
 By  \cite[Lemma 2.5]{wang_regularity_2000} applied to $u^+$ and the operator $\puccin$,  there exists $\delta>0$ such that  $\partial_{x_n} u^+ \geq 0$ in the set $\C_{\delta} \cap \{x_n > g(x')\}$. Also, applying the same Lemma to $u^-$ and  the operator $F(u)=-\puccip(-u) $,  we have  that $\partial_{-x_n} u^-=-\partial_{x_n} u^- \geq 0$ on the set $\mathcal{C}_{\delta} \cap \{x_n < g(x')\}$. Thus, since $u=u^+-u^-$, we conclude that $u$ is monotone increasing in the direction of $e_n=(0,\ldots,0,1)$ in $\mathcal{C}_{\delta}$.
 The same is true for any direction $\tau$ in the cone determined by $L$, the Lipschitz constant   of $g$; that is, 
 let $\Gamma(\theta,e_n)$ be the cone with axis $e_n$ and semi-opening $\theta$ given by $\mathrm{cotan}\,\theta=L$, then
$u$ is monotone increasing in the direction of $\tau \in \Gamma(\theta,e_n),$ in $ \mathcal{C}_{\delta}$.
$\Gamma(\theta,e_n)$ is called the monotonicity cone of $u$. 
 \medskip

 {\bf Step 2: Improvement of the Lipschitz regularity away from the free boundary.}
 
We may suppose that the monotonicity cone  exists for all points $x\in \mathcal{C}_1$, by using, if necessary, the invariance by elliptic dilation of the problem.
The monotonicity of $u$  along the directions of  $ \Gamma\left(\theta,e_n\right)$ implies that  for every small $\tau\in\Gamma\left(\frac{\theta}2,e_n\right)$,
\begin{equation}\label{conemonotonicity}\sup_{z\in B_\ep(x)}u(z-\tau)\leq u(x),
\end{equation}
 for every $x\in \mathcal{C}_{1-\ep}$, where $\ep=|\tau|\sin\left(\frac{\theta}2\right). $
 Let $x_0:=\frac34 e_n\in \mathcal{C}_1$. The proof of Lemma 4.6 of  \cite{caffarelli_geometric_2005} which uses Harnack inequality and Schauder estimates, can be adapted to our case 
 to improve the opening of the monotonicity cone in a neighborhood of $x_0$. The result goes as follows:  there exist positive constants $b$ and $c$  such that for every small $\tau\in \Gamma\left(\frac{\theta}2,e_n\right)$ and every $x\in B_\frac{1}{8}(x_0)$
 \begin{equation}\label{improvcone}\sup_{z\in B_{(1+ b)\ep}(x)}u(z-\tau)\leq u(x)-c\ep u(x_0),\end{equation}
with $\ep=|\tau|\sin\left(\frac{\theta}2\right).$
  \medskip
  
 {\bf Step 3: Construction of a family of subsolutions of variable radii.}
 
Here the main technique is the sup-convolution method to construct a family of subsolutions of the form 
$$w_\var(x)=\sup_{z\in B_{\var(x)}(x)}u(z -\tau),$$ for small $\tau\in \Gamma\left(\frac{\theta}2,e_n\right)$,   to compare with the solution $u$ of \eqref{freebddproblem}.
In order to apply the comparison principle, it is necessary to study the properties of the sup-convolution function and 
since problem \eqref{freebddproblem} is invariant by translations, it is enough to do it before translations, that is with $u(\cdot -\tau)$ replaced by $u$. 

For $0<r\le\frac18$,  $0<h<1$, there exists a family of functions  $\var_t$, $0\leq t\leq 1$,  with 
$\var_t\in C^2\left( \overline{B_1(0)}\setminus B_{\frac r2}(x_0)\right)$, $x_0=\frac34 e_n$,
 with the following properties:
 \begin{itemize}
\item[(a)] $1\leq \var_t\leq 1+t h$,
\item [(b)] $\varphi_t\equiv 1$ outside $B_\frac{7}{8}(0)$,
\item [(c)] $\varphi_t\geq 1+\lambda t h, $ in $B_\frac{1}{2}(0),$ for some $\lambda=\lambda(r)$,
\item [(d)] $|\nabla \varphi_t|\le Cth$.  
\end{itemize}
Moreover, if we define 
 \begin{equation}
 \label{before}
 v_{\var_t}(x):=\sup_{z\in B_{{\var_t}(x)}(x)} u (z),
 \end{equation}
\begin{itemize}
\item[(e)] then   
\begin{equation*}
 \puccin \left( v_{\var_t}\right) \geq 0\quad \textnormal{in } \Omega(v_{\var_t}^+),
 \end{equation*} 
 \begin{equation*}
 \puccip \left( v_{\var_t}\right) \geq 0\quad \textnormal{in } \Omega(v_{\var_t}^-),
 \end{equation*} 
 \end{itemize}
  and if  $|\nabla \varphi_t|<1$ then
 \begin{itemize}
 \item [(f)] for every point of $\partial \Omega(v_{\var_t}^+)$ there is a tangent ball contained in  $ \Omega(v_{\var_t}^+)$,
 \item [(g)] for every point $x_1\in \partial \Omega(v_{\var_t}^+),\,$ there exists  $\overline{\alpha}$ such that 
 \begin{equation}\label{freeboundarycondvphi2}
  v_{\var_t}(x)\geq  \bar{\alpha} <x-x_1,\bar{\nu}> + o(|x-x_1|),
\end{equation} 
where $\bar{\nu}$ is the normal vector of $\partial \Omega(v_{\var_t}^+)$ pointing inward $\Omega(v_{\var_t}^+)$.
 \end{itemize}


Properties (a)-(e) are proven in     \cite[Lemmas 3.4, 3.5]{wang_regularity_2000}. Since in  \cite{wang_regularity_2000} only concave operators (like $\Fn$) are considered, for the second inequality in (e) we refer  to  \cite[ Proposition 1.1]{Fel} where more general operators, not necessary concave,  are taken into account. Property (f) is proven in  \cite[Lemma 4.9]{caffarelli_geometric_2005}. 
 Let us prove (g).
 Note that $u \leq  v_{\var_t}$, therefore $\Om(u^+)\subset \Omega(v_{\var_t}^+)$. Now,  let $x_1\in\partial \Omega(v_{\var_t}^+)$,
 then there exists $y_1 \in \partial \Omega(u^+)$ such that $ v_{\var_t}(x_1)=u(y_1)=0$. Note that we must have $y_1 \in \partial B_{\var_t(x_1)}(x_1)$. 
 Thus, $B_{\var_t(x_1)}(x_1)$ is tangent to $\partial \Omega(u^+)$ at $y_1$ contained in $\Omega(u^-)$ and according to Definition \ref{visco def}  we have that
 \begin{equation}
 u(y)=\alpha  <y-y_1, {\bf \nu}> +o(|y-y_1|),
 \end{equation} where $\nu$ is the unit normal vector to $\partial\Om(u^+)$ at $y_1$ pointing inward $\Om(u^+)$.
 If  $y=x+ {\var_t}(x){\bf \nu}$, since  $y_1=x_1+ {\var_t}(x_1){\bf \nu}$, we obtain the asymptotic 
behavior of $v_{\var_t}$ in a neighborhood of $x_1$:
\begin{eqnarray}
v_{{\var_t}}(x) &\geq & u(y) \nonumber\\
&=& \alpha  <x+ {\var_t}(x) \nu -x_1-{\var_t}(x_1)\nu, {\bf \nu}> + o(|x-x_1|) \nonumber \\
&=& \alpha  <x-x_1 +({\var_t}(x)-{\var_t}(x_1))\nu, {\bf \nu}> +o(|x-x_1|) \nonumber.
\end{eqnarray}
We replace ${\var_t} (x)-{\var_t}(x_1)$ by $<x-x_1, \nabla {\var_t}(x_1)> + o(|x-x_1|)$ in the previous inequality and simplify to obtain
\begin{equation*}
v_{{\var_t}}(x) \geq \alpha <x-x_1,\nu + \nabla{\var_t} (x_1)> + o(|x-x_1|).
\end{equation*}
Thus, if we let
\begin{equation*}
\bar{\alpha}:=\alpha|\nu + \nabla {\var_t}(x_1)|, \quad \bar{\nu}:=\frac{\nu + \nabla {\var_t}(x_1)}{|\nu + \nabla {\var_t}(x_1)|},
\end{equation*} we obtain \eqref{freeboundarycondvphi2}. 
By  Lemma 4.9 in \cite{caffarelli_geometric_2005},   $\bar{\nu}$ is the unit normal vector to $\partial \Om^+(v_{{\var_t}})$ at $x_1$ pointing inward $\Om^+(v_{{\var_t}})$.
We note that in our problem we do not need the correctors used in the sup-convolution method to obtain the correct asymptotic behavior of $v_{{\var_t}}$ on points on the free boundary (see \cite[Lemma 4.12]{caffarelli_geometric_2005}).
\medskip

{\bf Step 4: Comparison with subsolutions.} 

In what follows, we will have to compare the solution $u$ of \eqref{freebddproblem}  with the functions
\begin{equation}\label{susolutionfamt}w_t(x):=\sup_{z\in B_{\ep\var_{bt}(x)}(x)}u(z -\tau),\quad x\in D,\end{equation}
  for small $\tau\in \Gamma\left(\frac{\theta}2,e_n\right)$, where $D:=B_\frac{9}{10}(0)\setminus B_\frac{1}{8}(x_0)$, 
 $b$ is defined in \eqref{improvcone}, $\ep=|\tau|\sin\left(\frac{\theta}2\right)$ and $\var_t$ is the family of functions defined in Step 3. By (d) in Step 3 we can choose $h$ small so that $\ep|\nabla\var_{bt}|<1$, therefore by (f), we have that 
\begin{equation}\label{regulapoinmtvarbt}
\text{for every point of }\partial \Omega(w_{\var_t}^+)\text{ there is  a tangent ball contained in }\Omega(w_{\var_t}^+).
\end{equation}
Now, having on hands \eqref{regulapoinmtvarbt} and  the asymptotic development \eqref{freeboundarycondvphi2}  we can  show the  following comparison result between $u$ and $w_t$: suppose that 
\begin{equation}\label{u>wt}u\geq w_t \text{ in }D,\quad u>w_t\quad \text{in }\Om(w_t^+),\text{ then }\partial\Om(w_t^+) \text{ and }\partial\Om(w_t^+)\text{ cannot touch}.\end{equation}

The proof is given in \cite[Lemma 2.1]{caffarelli_geometric_2005}. We perform it here for reader's convenience.
By \eqref{u>wt}, we know that $\Om(w_t^+)\subset\Om(u^+)$. Suppose by contradiction that  there exists 
$x_1 \in \partial\Om(w_t^+)\cap\partial\Om(u^+)$, then, by \eqref{regulapoinmtvarbt}, there exists a tangent  ball  to $\partial \Om(u^+)$ at $x_1$ contained in $\Om(u^+)$ . Thus,  according to  Definition \ref{freebddproblem}, we have  
\begin{equation}
\label{aqui}
u(x)=\alpha  <x-x_1, {\bf \nu}>+o(|x-x_1|).
\end{equation}
and by \eqref{freeboundarycondvphi2}, there exists $\eta >0$ such that 
\begin{equation}
\label{alla}
w_t(x) \geq \eta <x-x_1,\nu> + o(|x-x_1|).
\end{equation}
Note that here $\bar{\nu}=\nu$. 
Since $w_t \leq u$ and $w_t(x_1)=u(x_1)=0$, by \eqref{aqui} and \eqref{alla},  it follows  that 
\begin{equation}
\label{dolordecabeza}
\alpha=\eta.
\end{equation} 
We have that $u-w_t$ is a supersolution for $\puccin$ in $\Omega(w_t^+)$, since  by (c) in Proposition \ref{Fnproperties}, \eqref{freebddproblem} and 
(e), in 
$\Om(w_t^+)\subset\Om(u^+)$ we have
$$  0=\puccin(u)\geq  \puccin(u-w_t)+\puccin(w_t)\geq  \puccin(u-w_t).$$
Since $u > w_t$ in $\Omega(w_t)$, by the Hopf principle there exists $\delta>0$ such that 
\begin{equation*}
(u-w_t)(x_1+h\nu) \ge \delta h,
\end{equation*}  for all small $h>0$. This is a contradiction, since by \eqref{aqui},  \eqref{alla} and \eqref{dolordecabeza}, we have that
\begin{equation*}
(u-w_t)(x
_1+h\nu) \le o(h).
\end{equation*}
 Thus, we conclude that $\partial\Om(w_t^+)$ and $\partial\Om(u^+)$ cannot touch. 
\medskip


{\bf Step 5:  Carrying the improvement of Step 2 to the free boundary.}

 The improvement obtained in Step  2 needs to be carried to the free boundary, in $B_{1/2}(0)$, giving up a little bit of the interior improvement.  
 
 In order to do this, we consider the family of functions $w_t$ defined in \eqref{susolutionfamt}. Let $D:=B_\frac{9}{10}(0)\setminus B_\frac{1}{8}(x_0)$,
 let us check that the following conditions are satisfied:
 \begin{itemize}
 \item[i)] $w_0\le u$ in $D$,
 \item[ii)] $w_t\le u$ on $\partial D$ and $w_t<u$ in $\overline{\Om(w_t^+)}\cap\partial D$,
 \item[iii)] the family $\Om(w_t^+)$ is uniformly continuous, that is, for every $\ep>0$, 
 $$\Om(w_{t_1}^+)\subset \mathcal{N}_\ep( \Om(w_{t_2}^+))$$ 
 whenever $|t_1-t_2|<\delta(\ep)$, where $\mathcal{N}_\ep( \Om(w_{t_2}^+))$ is a $\ep$-neighborhood of $\Om(w_{t_2}^+)$.
 \end{itemize}
 
 By (a) in Step 3, $\varphi_0\equiv1$ and thus by \eqref{conemonotonicity}, if $x\in D$,  we have
\begin{equation} 
\label{initial}
 w_0(x)=\sup_{z\in B_\ep(x)} u(z-\tau)\le u(x),
\end{equation}
which is (i).

 By (b) in Step 3, and \eqref{conemonotonicity} if $x\in\partial B_\frac{9}{10}(0)$,  then 
  \begin{equation}\label{compawtupartB1}w_t(x)=\sup_{z\in B_\ep(x)} u(z-\tau)\le u(x),\end{equation} and the inequality is strict in
  $\overline{\Om(w_t^+)}$, by  taking  any $\ep'<\ep$ if necessary.
  If $x\in\partial B_\frac{1}{8}(x_0)$ by  (a) of Step 3 and \eqref{improvcone}, we have that (since $t,\,h\le1$), 
  \begin{equation}
  \label{compawtupartbx0}w_t(x)\leq\sup_{z\in B_{(1+tbh)\ep}(x)}u(z-\tau)\leq \sup_{z\in B_{(1+b)\ep}(x)}u(z-\tau)< u(x). \end{equation}
  Combining \eqref{compawtupartB1} and \eqref{compawtupartbx0} yields (ii).
  
 
 Finally, (iii) follows from the definition of the functions $w_t$, \eqref{susolutionfamt}.

Now, from (i)-(iii) and by using \eqref{u>wt}, we can conclude that
  \begin{equation}\label{maincomparison}w_t\le u\text{ in } D\text{ for every } t\in [0,1].\end{equation}
  The proof of \eqref{maincomparison} is given  in  \cite[Theorem 2.2]{caffarelli_geometric_2005} in the case of   the Laplace operator and we present it here for the sake of completeness. For that, let $E:=\{t\in[0,1]\,|\,v_t\le u\text{ in }\overline{D}\}$. By (i)  $0\in E$. $E$ is obviously closed. Let us show that it is open.
  If $t_0\in E$, that is $v_{t_0}\leq u$ in $D$, from (ii) and the strong maximum principle it follows that  $v_{t_0}< u$ in $\Om(v_{t_0}^+)\cap D$. 
  By  (ii) and \eqref{u>wt} we have that $\overline{\Om(v_{t_0})}\cap  D$ is compactly supported in $\Om(u^+)\cap D$ up to the boundary of $D$.
  From (iii), there exists $\delta>0$ such that $\overline{\Om(v_{t})}\cap D$ is compactly supported in $\Om(u^+)\cap D$ for all $t$ such that $|t-t_0|<\delta$.
  Thus, for such values of $t$, by (ii) and (e) of Claim 1 we have
  $$ \puccin \left( v_{\var_t}\right) \geq 0=\puccin(u)\quad \textnormal{in } \Omega(v_{\var_t}^+)\cap D,$$
  $$ v_{\var_t}\le u \quad\textnormal{on } \partial(\Omega(v_{\var_t}^+)\cap D)$$ and by the comparison principle,
  $v_{\var_t}\le u$ in $ \Omega(v_{\var_t}^+)\cap D$. Similarly, since 
   $$ \puccip \left( v_{\var_t}\right) \geq 0=\puccip(u)\quad \textnormal{in } \Omega(u^-)\cap D,$$
   and $$ v_{\var_t}\le u\quad \textnormal{on } \partial(\Omega(u^-)\cap D), $$ we have that $v_{\var_t}\le u$ in $ \Omega(u^-)\cap D$.
   Clearly $v_{\var_t}\le 0\le u$ in $\overline{\Om(u^+)}\cap \overline{ \Omega(v_{\var_t}^-)}\cap D$. We conclude that $v_{\var_t}\le u$ in $D$ and the openness of $E$ follows.
   Since $E$ is both an open and closed nonempty subset of $[0,1]$, we must have $E=[0,1]$. This proves \eqref{maincomparison}.

 Inequality \eqref{maincomparison} holds in particular for $t=1$ and hence using (c) in Step 3 we obtain that, on $B_{1/2}(0)$,
\begin{eqnarray*}
 u &\geq& w_{1} \\
    &=& \sup_{z\in B_{\ep \var_{b}(x)}(x)}u(z-\tau)\\
    &\geq& \sup_{z\in B_{\ep (1+ (\lambda h) b)}(x)}u(z-\tau),
\end{eqnarray*} 
which implies the desired improvement of the cone of monotonicity across the free boundary. The original radius $\ep$ in \eqref{initial} was first improved to $\ep + \ep b$ far from the free boundary (see \eqref{improvcone}),
and at the free boundary the radius became $\ep + (\lambda h) \ep b$. Since $\lambda h <1$, a little bit of opening in the cone has to be given up in order to bring the improvement across the free boundary (see 
Theorem 4.2 and Lemma 4.4 in \cite{caffarelli_geometric_2005} for details).
\medskip

{\bf Step 6: Basic iteration.}
 
Rescaling and repeating  Steps 2-5 we obtain that the free boundary is $C^{1,\alpha}$ in $\mathcal{C}_\frac{1}{2}$, see the proof of Theorem 4.1 in  \cite{caffarelli_geometric_2005} for details.

\end{proof}

\subsection{Flat free boundaries are Lipschitz}
In this subsection we prove that 
if $u$
is a solution of the free boundary problem \eqref{freebddproblem} and the free boundary can be trapped in a narrow neighborhood in between two Lipschitz graphs, then the free boundary is actually Lipschitz.
Let us recall the definition of $\ep$-monotone function. 
\begin{defn}\label{uepmondef1}
We say that $u$ is $\ep$-monotone in the cylinder $\mathcal{C}_1$ along a direction $\tau$, with $|\tau|=1$, if for all $x\in \mathcal{C}_1$, 
$$u(x+l\tau)\ge u(x),$$ for all $l\geq \ep$ such that  $x+l\tau\in \mathcal{C}_1$.
\end{defn}
The $\ep$-monotonicity can be reformulated equivalently as follows, see \cite{caffarelli_geometric_2005}. 

\begin{defn}\label{uepmondef2}
We say that $u$ is $\ep_0$-monotone   in the cylinder $\mathcal{C}_1$ along the directions of  the cone $\Gamma(\theta,e)$ if  for all $x\in \mathcal{C}_1$,
$$\sup_{y\in B_{\ep\sin \theta}(x)}u(y-\ep e)\leq u(x),$$
for any $\ep\ge\ep_0$ such that $B_{\ep\sin \theta }(x-\ep e)\subset \mathcal{C}_1$.
\end{defn}
As in Subsection \ref{C1alphasec}, in the definition above $\Gamma(\theta,e)$ denotes the cone of semi-opening $\theta$ and axis $e$.
\begin{remark}\label{lipremtrap}
If  $u$ is $\ep$-monotone in $\mathcal{C}_1$ according to Definition \ref{uepmondef2}, then the level surfaces  of $u$ 
in $\C_1$, $\partial\{u>t\}$,  are contained in a $(1-\sin \theta)\ep $ size of the graph of a Lipschitz function $g$ with Lipschitz constant $L= \mathrm{cotan }\, \theta <1$, 
see \cite{caffarelli_geometric_2005}. 
\end{remark}

\begin{prop} 
\label{Caffa2} Let $\frac{\pi}{4}<\theta<\frac{\pi}{2}$ and let $u$ be a viscosity solution of the problem \eqref{freebddproblem} in 
$\mathcal{C}_1=B_1'\times(-1,1)$. Assume that $0 \in \partial \Om(u^+)$. Then there exists $\ep=\ep(\theta)$ such that if $u$ is $\ep$-monotone in 
$\mathcal{C}_{1-\ep}=B_{1-\ep}'\times(-1+\ep,1-\ep)$ along any direction $\tau$ in the cone $\Gamma(\theta,e)$, then $u$ is fully monotone in 
$\mathcal{C}_\frac12=B_\frac12'\times\left(-\frac12,\frac12\right)$ along any direction $\tau\in \Gamma(\theta_1,e)$ with $\theta_1=\theta_1(\theta,\ep)$.
\end{prop}

\begin{proof}

 The proof of this result follows from 
 \cite{caffarelli_harnack_1987} (see also \cite{caffarelli_geometric_2005}) and 
  \cite{wang_regularity_2000}. 
 We will sketch the proof below.
 
  \medskip

{\bf Step 1: Full monotonicity of $u$ outside a strip of size $M\ep$ of the free boundary.} 

 By Lemma 1 in \cite{wang_regularity_2002} there exists $M>1$ such that in $\mathcal{C}_1\setminus \mathcal{N}_{M\ep}$, where
 $$\mathcal{N}_{M\ep}:=\{x\in \mathcal{C}_1\,|\,d(x,\partial\Om(u^+))<M\ep\}$$ 
 $u$ is actually fully monotone along any direction of $\tau\in\Gamma(\theta,e)$. 
 
  \medskip

 {\bf Step 2: Construction of a family of subsolutions of variable radii.}

 Following  the method developed in \cite{caffarelli_harnack_1989}, we need to construct a family of subsolutions of the form 
$$w(x)=\sup_{z\in B_{\var(x)}(x)}u(z -\lambda \ep e),$$ for some $\lambda \in(0,1)$, to compare with the solution $u$ of \eqref{freebddproblem}. 
Up to a change of coordinates, we can assume that 
$$e=e_n.$$
 Since $u$ is $\ep$-monotone, by Remark \ref{lipremtrap}
there exists $g:\R^{n-1}\to\R$ with $g(0)=0$ and Lipschitz constant $L= \mathrm{cotan} \, \theta <1$, such that if
\begin{equation}\label{A}A:=\{(x',x_n)\in\R^n\,|\,x_n=g(x')\},\end{equation} then 
\begin{equation}\label{partialonflat}\partial\Om(u^+)\subset \mathcal{N}_\ep(A),\end{equation}
where 
$$\mathcal{N}_\ep(A):=\{x\in \mathcal{C}_1\,|\,d(x,A)<\ep\}.$$
By Lemmas 2 and  3 in \cite{wang_regularity_2002} and Proposition 1.1 in \cite{Fel}, for any given $\delta>0$, there exists a family of $C^2$-functions,   $\var_t$, $0\leq t\leq 1$, 
defined on $ \mathcal{C}:=\overline{B_1'(0)}\times[-2L,2L]$, 
 with the following properties:
\begin{itemize}
\item[a)] $1\leq \var_t\leq 1+t $,
\item [b)] $\varphi_t\equiv 1$ on  $A_\delta:=\{x\in \mathcal{C}\,|\,d(x,A\cap \partial\mathcal{C})<\delta\}$,
\item [c)] in the set $\{x\in \mathcal{C}\,|\,d(x,\partial \mathcal{C})>\delta\}$,
$$\varphi_t\geq 1+t\left(1-\frac{C\delta}{d(x,\partial \mathcal{C})^2}\right),$$
\item [d)] $|\nabla \varphi_t|\le \frac{Ct}{\delta}.$
\end{itemize}
Moreover,
\begin{itemize}
\item[e)] 
 if we define 
 \begin{equation*}
 v_{\var_t}(x):=\sup_{z\in B_{{\var_t}(x)}(x)} u (z),
 \end{equation*}
  then $v_{\var_t}$  satisfies
  \begin{equation*}
 \puccin \left( v_{\var_t}\right) \geq 0\quad \textnormal{in } \Omega(v_{\var_t}^+),
 \end{equation*} 
  \begin{equation*}
 \puccip \left( v_{\var_t}\right) \geq 0\quad \textnormal{in } \Omega(v_{\var_t}^-),
  \end{equation*} 
 and if $|\nabla \varphi_t|<1$ then
 \item [f)] for every point of $\partial \Omega^+(v_{\var_t})$ there is  a tangent ball contained in $\Omega^+(v_{\var_t})$, 
  \item [g)] if $$0<\sin \overline{\theta} \leq \frac{1}{1+|\nabla\varphi_t|}\left(\sin\theta-\frac{\ep}{2\varphi_t}\cos^2\theta-|\nabla\varphi_t|\right),$$
 then $ v_{\var_t}$ is monotone in the cone $\Gamma(\overline{\theta},e_n)$; in particular its level surfaces are Lipschitz graphs, in the direction of $e_n$, with Lipschitz constant 
 $\overline{L}\leq \mathrm{cotan}\,\overline{\theta}$.
 \end{itemize}
 Finally, as in the proof of Proposition \ref{Caffa1},  if $|\nabla \varphi_t|<1$, the function $v_{\var_t}$ has the following behavior at points of $\partial \Omega(v_{\var_t}^+)$
 \begin{itemize}
\item[h)] for every point $x_1\in \partial \Omega(v_{\var_t}^+)$ there exists  $\bar{\alpha}>0$ such that 
 $$ v_{\var_t}(x)\geq  \bar{\alpha} <x-x_1,\bar{\nu}> + o(|x-x_1|),$$
where
$\bar{\nu}$ is the normal vector of $\partial \Omega(v_{\var_t}^+)$ pointing inward $\Omega(v_{\var_t}^+)$. 
\end{itemize}
\medskip

{\bf Step 3: Comparison with subsolutions.} 
In what follows, we will have to compare the solution $u$ of \eqref{freebddproblem}  with the functions
\begin{equation}\label{susolutionfamtflat}w_t(x):=\sup_{z\in B_{\sigma\var_{t}(x)}(x)}u(z -\lambda\ep e_n),\end{equation}
for  $\sigma,\,\lambda\in(0,1)$ to be determined, where $\var_t$ is the family of functions defined in Step 2.
We first notice that from the $\ep$-monotonicity of $u$ (Definition \ref{uepmondef2}), for $1-\lambda<\sqrt{2}/2$, we have
\begin{equation}\label{compasubflat1}\sup_{z\in B_{\ep(\sin \theta-(1-\lambda))}(x)} u(z -\lambda\ep e_n)\leq 
\sup_{z\in B_{\ep\sin \theta}(x)} u(z -\ep e_n)\leq
u(x),\end{equation}
since $B_{\ep(\sin \theta-(1-\lambda))}(x -\lambda\ep e_n)\subset B_{\ep\sin \theta}(x -\ep e_n)$.

For any $\eta>0$ and $A$ defined as in \eqref{A}, let us denote by $\mathcal{N}_\eta(A)$ the $\eta$-neighborhood of $A$, defined by
$$ \mathcal{N}_\eta(A):=\{x\in\mathcal{C}\,|\,d(x, A)<\eta\}.$$
By Step 1 and \eqref{partialonflat}, $u$ is fully monotone in the directions of $\Gamma(\theta,e_n)$, outside the set  $\mathcal{N}_{2M\ep}(A)$. Therefore, 
\begin{equation}\label{compasubflat2}\sup_{z\in B_{\lambda\ep\sin \theta }(x)} u(z -\lambda\ep e_n)\leq u(x)\quad\text{for }x\not\in\mathcal{N}_{2M\ep}(A).\end{equation}
We now choose
\begin{equation}\label{compasubflat3}\sigma:=\ep(\sin \theta-(1-\lambda)),\quad \lambda\ge \frac32-\frac{\sqrt{2}}{2},\quad \delta:=\ep^\frac12.\end{equation}
Then the family of functions $w_t$ in \eqref{susolutionfamtflat} is well defined in $\mathcal{C}_{1-\ep}\cap \mathcal{N}_{2M\ep}(A) $.
Moreover, (e)-(h) of Step 2 hold true for $\ep$ (and thus $\sigma$) small enough.
Since $\sigma$ defined as in \eqref{compasubflat3} satisfies $\sigma<\lambda \ep\sin \theta$, by (a) of Step 2 we can choose $\overline{t}>0$ so small that
\begin{equation}\label{compasubflat3bis}\sigma\varphi_t\leq \lambda \ep\sin \theta,\quad\text{for }0\le t\le \overline{t}.\end{equation}
By (e)-(h) of   Step 2, the functions $w_t$, $0\le t\le 1$, satisfy
\begin{equation}\label{compasubflat4bis} \puccin \left( w_t\right) \geq 0\quad \textnormal{in } \Omega(w_{t}^+),\end{equation}
\begin{equation}\label{compasubflat4} \text{ for any point of $\partial \Omega(w_t^+)$ there is a tangent ball contained in }\Omega(w_t^+)\end{equation}
 \begin{equation}\begin{split}\label{compasubflat5}& \text{For every point }x_1\in \partial \Omega(w_{t}^+),\,\text{ there exists  }\bar{\alpha}>0 \text{ such that } \\&
  w_{t}(x)\geq  \bar{\alpha} <x-x_1,\bar{\nu}> + o(|x-x_1|).
\end{split}  \end{equation} 
Let us show that for all $0\le t\le \overline{t}$, 
 \begin{equation}\label{compasubflat6}w_t(x)\leq u(x)\quad \text{for }x\in \partial (\mathcal{N}_{2M\ep}(A)\cap \mathcal{C}_{1-4\ep}).
 \end{equation} 
 If $x\in \partial (\mathcal{N}_{2M\ep}(A))\cap \mathcal{C}_{1-4\ep}$, then by \eqref{compasubflat3bis} and \eqref{compasubflat2}, we have that
 \begin{equation}\label{compasubflat7}w_t(x)\leq \sup_{z\in B_{\lambda\ep\sin \theta}(x)} u(z -\lambda\ep e_n)\leq u(x).\end{equation}
 If $x\in  \mathcal{N}_{2M\ep}(A))\cap  \partial (\mathcal{C}_{1-4\ep})$, then, since for $\ep$ small enough $\delta=\ep^{1/2}>4\ep$, by (b) of Step 2, $\varphi_t(x)=1$. Thus, by the definition of $\sigma$ in \eqref{compasubflat3} and \eqref{compasubflat1}, for $x\in  \mathcal{N}_{2M\ep}(A))\cap  \partial (\mathcal{C}_{1-4\ep})$,
 $$w_t(x)=\sup_{z\in B_{\ep(\sin \theta-(1-\lambda))}(x)} u(z -\lambda\ep e_n)\leq u(x).$$
 This concludes the proof of \eqref{compasubflat6}.
 
Finally, by \eqref{compasubflat6} and using that the functions $w_t$ satisfy \eqref{compasubflat4bis}-\eqref{compasubflat5}, arguing as in Step 5 of the proof of Proposition \ref{Caffa1}, we infer that, for 
$0\le t\le\overline{t}$,
 \begin{equation}\label{compasubflat8} w_t(x)\leq u(x)\quad \text{for all }x\in \mathcal{N}_{2M\ep}(A)\cap \mathcal{C}_{1-4\ep}.
 \end{equation} 

\medskip

{\bf Step 4: From the $\ep$-monotonicity to the $\lambda\ep$-monotonicity.} 

Arguing as in  \cite{caffarelli_harnack_1989} (see also Lemma 5.7 in \cite{caffarelli_geometric_2005}), by \eqref{compasubflat8}  and 
 (c) of Step 2, we have that there exists $c_0>0$ such that in $ \mathcal{N}_{2M\ep}(A)\cap \mathcal{C}_{1-4\ep^{1/8}}$
 $$\sup_ {\lambda\ep\sin(\theta-c_0\ep^{1/4})}u(z -\lambda\ep e_n)\leq u(x),$$
 that is $u$ is $\lambda\ep$-monotone in any direction of the cone of directions $\Gamma(\theta-c_0\ep^{1/4},e_n)$. 
 
  \medskip
  
  {\bf Step 5: Basic iteration. }
  
  Rescaling and repeating  Steps 1-4, we obtain that the free boundary is Lipschitz in $\C_\frac{1}{2}$, see the proof of Theorem 5.1 in  \cite{caffarelli_geometric_2005} for details.

\end{proof}

\subsection{Proof of Theorem \ref{C1alphaGammathm}}
Let $u$ be a solution of  \eqref{limitproblemk=2}. 
Then, by Theorem \ref{limitpbthm}, $u$ is a solution of the free boundary problem \eqref{freebddproblem} in the sense of Definition \ref{visco def}.
Let $z\in\Gamma$ be a regular point. Assume without loss of generality that $z=0$. 
By Corollary \ref{flatnessprop}, there exists $r_j\to0$  as $j\to+\infty$ with the following property: for any $\ep>0$ there exists $J\in\N$ such that for any $j\ge J$,  all the level sets of $u_{r_j}(x)=u(r_j x)/r_j$ in $B_2(0)$ are $\ep$-flat. Also, by scaling invariance
$u_{r_j}$ is  solution of  \eqref{freebddproblem} in the cylinder 
$\mathcal{C}_1=B'_1(0)\times (-1,1)$. We can now apply Propositions \ref{Caffa1} and \ref{Caffa2} to conclude that there is $J\in\N$ such that for any $j\ge J$ the set $\partial\,\Om((u_1)_{r_j})\cap B_\frac{1}{4}(0)$ is of class $C^{1,\alpha}$  for some $0<\alpha\le 1$. Therefore, the same is true for $\Gamma\cap B_\frac{r_j}{4}(0)$, as $\Gamma\cap B_\frac{r_j}{4}(0)= r_j\partial \Om((u_1)_{r_j}))\cap B_\frac{1}{4}(0))$. Let us prove that the set of regular points is open in $\Gamma$. 

By the elliptic regularity theory, see Corollary 1.8 in \cite{MR2853528}, $u_1 \in C^{1,\alpha}(\overline{\Om(u_1)}\cap  B_\frac{r_j}{8}(0))$ and $u_2\in C^{1,\alpha}(\overline{\Om(u_2)}\cap  B_\frac{r_j}{8}(0))$, thus
\begin{equation}
u(x)= \frac{\partial u_1}{\partial \nu_1}(0) <x, {\bf \nu}>^+ - \,\frac{\partial u_2}{\partial \nu_2}(0) <x,{\bf \nu}>^-+o(|x|),
\end{equation}
 and by Theorem \ref{alpha=betathm}
$$\frac{\partial u_1}{\partial \nu_1}(0)=\frac{\partial u_2}{\partial \nu_2}(0)>0,$$ where $\nu_i$ is  the interior unit normal vector to $\Om(u_i)$. In particular, $u$ has the asymptotic behavior \eqref{asymptodevumainthm} at 0.
By the $C^{1,\alpha}$ local regularity of $u_1$ and $u_2$ up to the free boundary , there exists $s<r_j/8$, such that: 
\begin{equation}
\frac{\partial u_1}{\partial \nu_1}(x_0)>0, 
\quad \frac{\partial u_2}{\partial \nu_2}(x_0)>0, \text{ for any } x_0\in \Gamma\cap B_s(0), 
\end{equation}
and
\begin{equation*}
u(x)= \frac{\partial u_1}{\partial \nu_1}(x_0) <x-x_0, {\bf \nu}>^+ - \,\frac{\partial u_2}{\partial \nu_2}(x_0) <x-x_0,{\bf \nu}>^-+o(|x-x_0|).
\end{equation*}
Hence each $x_0\in \Gamma\cap B_s(0)$ is a regular point of $u$. Actually, again from Theorem \ref{alpha=betathm}, we have that $\frac{\partial u_1}{\partial \nu_1}(x_0)=\frac{\partial u_2}{\partial \nu_2}(x_0)$. We have proven that the set of regular points is an open set of $\Gamma$, locally of class $C^{1,\alpha}$ and this concludes the proof of the theorem.
\section{Appendix}

\begin{lem}\label{fondamentalsollem}
Assume $r,\gamma,c>0$,  and let $$\psi(x)=c\left(\frac{r^\gamma}{|x|^\gamma}-1\right),\quad x\neq 0.$$
Then, the following holds.

\begin{itemize}
\item[i)] $\psi(x)>0$ if $|x|<r$, $\psi(x)=0$ if $|x|=r$, $\psi(x)<0$ if $|x|>r$.
\item[ii)] If $\nu$ is the interior normal unit vector  of $B_r(0)$, then

$$\nabla\psi(x)=\frac{c\gamma}{r}\nu\quad\text{for any }x\in\partial B_r(0).$$
\item[iii)] For any $x\in B_r(0)$,
$$\psi(x)\geq \frac{c\gamma}{r}(r-|x|).$$
\item[iv)] If $\gamma=\frac{\Lambda(n-1)-\lambda}{\lambda}$, then $\ppuccin(\psi)(x)=0$ for all $x\neq 0$.
\end{itemize}
\end{lem}
\begin{proof}
Property (i) is immediate. 

To prove (ii)-(iv), let us compute the gradient and the Hessian matrix of $\psi$. We get, for $x\neq0$,
$$\nabla\psi(x)=-c\gamma r^\gamma\frac{x}{|x|^{\gamma+2}},$$
and 
$$D^2\psi(x)=\frac{c\gamma r^\gamma}{|x|^{\gamma+2}}\left((\gamma+2)\frac{x\otimes x}{|x|^2}-I_n\right),$$
where $I_n$ is the $n\times n$ identity matrix. 

In particular, if  $|x|=r$ and  $\nu=-\frac{x}{r}$ is the interior normal unit vector of $ B_r(0)$ at $x$, then we see that 
$$\nabla\psi(x)=-\frac{c\gamma}{r}\frac{x}{r}=\frac{c\gamma}{r}\nu,$$
which proves (ii).

To prove (iii), let us denote $\rho=|x|$ and let $\psi(\rho)=c\left(\frac{r^\gamma}{\rho^\gamma}-1\right)$. Then using that $\psi'(r)=-\frac{c\gamma}{r}$ and that $\psi''(\rho)\ge 0$,
we get
$$\psi(\rho)\ge \frac{c\gamma}{r}(r-\rho),$$ which gives (iii). 

Next, it is easy to see that, given any $n\times n$-matrix  $A$ with eigenvalues  $\lambda_1,\ldots,\lambda_n$, then the eigenvalues of $A-I_n$ are 
$\lambda_1-1,\ldots,\lambda_n-1$. Therefore, since the eigenvalues of $\frac{x\otimes x}{|x|^2}$ are $\lambda_1=\ldots=\lambda_{n-1}=0$ and $\lambda_n=1$, we infer that $(\gamma+2)\frac{x\otimes x}{|x|^2}-I_n$ has  $(n-1)$ negative eigenvalues equal to $-1$ and one positive eigenvalue equal to $(\gamma +1)$.
In particular $$\ppuccin(\psi)=\frac{c\gamma r^\gamma}{|x|^{\gamma+2}}\left[\lambda (\gamma+1)-\Lambda(n-1)\right].$$ 
Property  (iv)   then follows.

\end{proof}

\begin{lem}\label{fondamentalsollemF}
 Let $\phi$ be the solution of
\begin{equation}\label{phiappendixlem}
\begin{cases}
\Fn(\phi)=0&\text{in }B_r(0)\setminus B_\frac{r}{2}(0)\\
\phi=1&\text{on }\partial B_\frac{r}{2}(0)\\
\phi=0&\text{on }\partial B_r(0).
\end{cases}
\end{equation}
 Then, $\phi=\phi(|x|)$ is a radial function and there exists a constant $\sigma>0$ independent of $r$ such that for $x\in B_r(0)\setminus B_\frac{r}{2}(0)$ and $y_0\in\partial B_r(0)$, 
\begin{equation*} \phi(x)=\frac{\sigma}{r}<x-y_0,\nu>+o(|x-y_0|),
\end{equation*}
where $\nu$ is the interior normal unit vector  of $B_r(0)$ at $y_0$.
\end{lem}
\begin{proof}
Let $\varphi$ be the solution of \eqref{phiappendixlem} with $r=1$. Then, since $\Fn$ is a concave operator, we have that $\varphi\in C^{2,\alpha}(\overline{B_1(0)}\setminus B_\frac{r}{2}(0))$, see \cite{caffarelli_cabre_1995}. Let $O$ be any orthogonal matrix and let $v(x):=\varphi(Ox)$. By Proposition \ref{orthogonalF}, $\Fn$ is invariant under rotations, thus $v$ is solution of 
\eqref{phiappendixlem}  and by uniqueness, $\varphi(Ox)=\varphi(x)$.  Since the latter equality holds true for any orthogonal matrix  $O$, we infer that $\varphi$ is a radial function, $\varphi=\varphi(|x|)$. 

Let $\psi_1(x):=1/(2^\gamma-1)\left(\frac{1}{|x|^\gamma}-1\right)$ where $\gamma=\frac{\Lambda(n-1)-\lambda}{\lambda}$, and let
$\psi_2$ be the harmonic function solution of 
\begin{equation*}
\begin{cases}
\Delta \psi_2=0&\text{in }B_1(0)\setminus B_\frac{1}{2}(0)\\
\psi_2=1&\text{on }\partial B_\frac{1}{2}(0)\\
\psi_2=0&\text{on }\partial B_1(0),
\end{cases}
\end{equation*}
 i.e., for $n>2$, $\psi_2(x)=1/(2^{n-2}-1)\left(\frac{1}{|x|^{n-2}}-1\right).$ 
Then by Lemma \ref{fondamentalsollem} and the comparison principle, 
for $x\in B_1(0)\setminus B_\frac{1}{2}(0)$, 
$$\psi_1(x)\le \varphi(x)\le \psi_2(x)$$
and thus  there exists $\sigma, $ $\gamma/(2^\gamma-1)\le\sigma \le (n-2)/(2^{n-2}-1)$,  such that  if  $y_0\in\partial B_1(0)$,
$$\varphi(x)=\sigma <x-y_0,\nu>+o(|x-y_0|).$$
The lemma is proven by noticing that $\phi(x)=\varphi(x/r)$ is the solution of \eqref{phiappendixlem}.
\end{proof}

\bibliography{CPQT}

\begin{thebibliography}{10}

\bibitem{MR2511639}
R.~Argiolas and F.~Ferrari.
\newblock Flat free boundaries regularity in two-phase problems for a class of
  fully nonlinear elliptic operators with variable coefficients.
\newblock {\em Interfaces Free Bound.}, 11(2), 2009.

\bibitem{caffarelli_cabre_1995}
L.~Caffarelli and X.~Cabr\'{e}.
\newblock {\em Fully nonlinear elliptic equations}, volume~43.
\newblock American Mathematical Society, 1995.

\bibitem{caffarelli_geometric_2005}
L.~Caffarelli and S.~Salsa.
\newblock {\em A geometric approach to free boundary problems}, volume~68.
\newblock American mathematical society Providence, RI, 2005.

\bibitem{caffarelli_harnack_1986}
Luis~A Caffarelli.
\newblock A harnack inequality approach to the regularity of free boundaries.
\newblock {\em Communications on Pure and Applied Mathematics}, 39(S1), 1986.

\bibitem{caffarelli_harnack_1987}
Luis~A Caffarelli.
\newblock A harnack inequality approach to the regularity of free boundaries.
  part i: Lipschitz free boundaries are $C^{1,\alpha}$.
\newblock {\em Revista Matem{\'a}tica Iberoamericana}, 3(2):139--162, 1987.

\bibitem{caffarelli_harnack_1989}
Luis~A Caffarelli.
\newblock A harnack inequality approach to the regularity of free boundaries
  part ii: Flat free boundaries are lipschitz.
\newblock {\em Communications on Pure and Applied Mathematics}, 42(1):55--78,
  1989.

\bibitem{cct}
Gui-Qiang Chen, Giovanni Comi, and Monica Torres.
\newblock The Gauss-Green formula for $ \mathcal{DM}^p$-fields on open sets.
\newblock {\em Preprint}, 2018.

\bibitem{ctz}
Gui-Qiang Chen, Monica Torres, and William Ziemer.
\newblock Gauss-Green theorem for weakly differentiable vector fields, sets of
  finite perimeter, and balance laws.
\newblock {\em Communications on Pure and Applied Mathematics},
  LVII:0242--0304, 2009.

\bibitem{crandall_user_1992}
Michael~G Crandall, Hitoshi Ishii, and Pierre-Louis Lions.
\newblock User's guide to viscosity solutions of second order partial
  differential equations.
\newblock {\em Bulletin of the American Mathematical Society}, 27(1):1--67,
  1992.

\bibitem{MR3310271}
D.~De~Silva, F.~Ferrari, and S.~Salsa.
\newblock Free boundary regularity for fully nonlinear non-homogeneous
  two-phase problems.
\newblock {\em J. Math. Pures Appl. (9)}, 103(3):658--694, 2015.

\bibitem{Fel}
M.~Feldman.
\newblock Regularity of lipschitz free boundaries in two-phase problems for
  fully nonlinear elliptic equations.
\newblock {\em Indiana Univ. Math. J.}, 50:1171--1200, 2001.

\bibitem{gilbarg_elliptic_2001}
D.~Gilbarg and N.S. Trudinger.
\newblock {\em Elliptic partial differential equations of second order}, volume
  224.
\newblock Springer Verlag, 2001.

\bibitem{han_elliptic_2011}
Q.~Han and F-H Lin.
\newblock {\em Elliptic partial differential equations}, volume~1.
\newblock American Mathematical Soc., 2011.

\bibitem{ishiilions}
H.~Ishii and P.-L. Lions.
\newblock Viscosity solutions of fully nonlinear second-order elliptic partial
  differential equations.
\newblock {\em J. Differential Equations}, 83(1):26--78, 1990.

\bibitem{jiang_zero_2004}
H.~Jiang and F-H Lin.
\newblock Zero set of sobolev functions with negative power of integrability.
\newblock {\em Chinese Annals of Mathematics}, 25(01):65--72, 2004.

\bibitem{krylov_controlled_2008}
Nikolaj~Vladimirovi{\v{c}} Krylov.
\newblock {\em Controlled diffusion processes}, volume~14.
\newblock Springer Science \& Business Media, 2008.

\bibitem{MR2853528}
F.~Ma and L.~Wang.
\newblock Boundary first order derivative estimates for fully nonlinear
  elliptic equations.
\newblock {\em J. Differential Equations}, 252(2):988--1002, 2012.

\bibitem{petrosyan_regularity_2012}
A.~Petrosyan, H.~Shahgholian, and N.~Uraltseva.
\newblock {\em Regularity of free boundaries in obstacle-type problems}, volume
  136.
\newblock American Mathematical Society Providence (RI), 2012.

\bibitem{quitalo_free_2013}
V.~Quitalo.
\newblock A free boundary problem arising from segregation of populations with
  high competition.
\newblock {\em Archive for Rational Mechanics and Analysis}, 210(3):857--908,
  2013.

\bibitem{wang_regularity_2000}
P-Y. Wang.
\newblock Regularity of free boundaries of two-phase problems for fully
  nonlinear elliptic equations of second order i. lipschitz free boundaries are
  $C^{1,\alpha}$.
\newblock {\em Communications on Pure and Applied Mathematics}, 53(7):799--810,
  2000.

\bibitem{wang_regularity_2002}
P-Y. Wang.
\newblock Regularity of free boundaries of two-phase problems for fully
  nonlinear elliptic equations of second order. ii. flat free boundaries are
  lipschitz.
\newblock {\em Communications in Partial Differential Equations},
  27(7-8):1497--1514, 2002.

\bibitem{MR2005161}
P-Y. Wang.
\newblock Existence of solutions of two-phase free boundary problems for fully
  nonlinear elliptic equations of second order.
\newblock {\em J. Geom. Anal.}, 13(4):715--738, 2003.

\end{thebibliography}
\bibliographystyle{plain}

\end{document}